\documentclass[a4paper,12pt]{amsart}
\usepackage[ngerman,english]{babel}
\usepackage[utf8]{inputenc}
\usepackage{amssymb}
\usepackage{amsmath}
\usepackage{amsfonts}
\usepackage{amstext}
\usepackage{amsthm}
\usepackage{todonotes}

\usepackage{parskip}

\usepackage{a4wide}
\usepackage{enumerate}
\usepackage{hyperref}

\newcommand{\IN}{{\mathbb{N}}}
\newcommand{\IR}{{\mathbb{R}}}
\newcommand{\IZ}{{\mathbb{Z}}}

\newcommand{\1}{\mathbbmss{1}}

\renewcommand{\phi}{\varphi}
\renewcommand{\epsilon}{\varepsilon}

\newcommand{\A}{\mathcal{A}}

\newcommand{\Tr}{\operatorname{Tr}}

\newcommand{\Qgen}{{\widetilde{Q}}}
\newcommand{\Dgen}{{\widetilde{D}}}
\newcommand{\LapUnw}{{\mathcal{L}}}
\newcommand{\LapDom}{{\widetilde{F}}}

\newcommand{\ltwo}{{\ell^2(X,m)}}
\newcommand{\linf}{{\ell^\infty(X)}}

\newcommand{\NeuDir}{{Q^{(N)}}}

\newcommand{\DirDir}{{Q^{(D)}}}

\newcommand{\QD}{{Q^{(D)}}}
\newcommand{\QN}{{Q^{(N)}}}
\newcommand{\QNe}{{Q^{(N)}_{\rm e}}}
\newcommand{\QDe}{{Q^{(D)}_{\rm e}}}
\newcommand{\Qe}{{Q_{\rm e}}}

\newcommand{\BFEspace}{\widetilde{D}\cap\ell^\infty(X)}

\newtheorem{theorem}{Theorem}[section]
\newtheorem{lemma}[theorem]{Lemma}
\newtheorem{proposition}[theorem]{Proposition}
\newtheorem{corollary}[theorem]{Corollary}

\theoremstyle{definition}
\newtheorem{definition}[theorem]{Definition}

\newtheorem{remarks}[theorem]{Remark}

\newcommand{\ow}[1]{\widetilde{ #1}}
\newcommand{\as}[1]{\left\langle #1\right\rangle}
\newcommand{\Dzero}{D_0}

\newcommand{\Hm}[1]{\leavevmode{\marginpar{\tiny%
$\hbox to 0mm{\hspace*{-0.5mm}$\leftarrow$\hss}%
\vcenter{\vrule depth 0.1mm height 0.1mm width \the\marginparwidth}%
\hbox to 0mm{\hss$\rightarrow$\hspace*{-0.5mm}}$\\\relax\raggedright
#1}}}

\begin{document}

\title{Boundary representation of Dirichlet forms on discrete spaces}

\author[Keller]{Matthias Keller}
\address{M. Keller, M. Schwarz, Institut für Mathematik \\Universit{\"a}t Potsdam \\14476 Potsdam, Germany } \email{matthias.keller@uni-potsdam.de, mschwarz@math.uni-potsdam.de}

\author[Lenz]{Daniel Lenz}
\address{D. Lenz, M. Schmidt, Mathematisches Institut \\Friedrich Schiller Universit{\"a}t Jena \\07743 Jena, Germany }\email{daniel.lenz@uni-jena.de, schmidt.marcel@uni-jena.de}

\author[Schmidt]{Marcel Schmidt}

\author[Schwarz]{Michael Schwarz}

\begin{abstract}
We describe the set of all  Dirichlet forms associated to a given
infinite graph in terms of Dirichlet forms on its Royden boundary.
Our approach is purely analytical and uses form methods.
\end{abstract}

\maketitle


\section*{Introduction}
Our aim in this paper is to characterize all Dirichlet forms that
are  associated to a given infinite weighted graph.  This problem is
motivated by the corresponding question  for bounded domains in
Euclidean space. For such domains  it is a most classical topic to
extend a symmetric elliptic differential  operator  to a
self-adjoint or even Markovian operator by posing suitable boundary
conditions on the geometric boundary. Indeed, such extensions play
an important role in all sorts of considerations not least as every
self-adjoint extension corresponds to a possible description of the
system in quantum mechanics and the smaller class of Markovian
self-adjoint extensions corresponds to descriptions of diffusion.
Accordingly, Markovian extension on bounded domains in Euclidean space
have received a lot of attention over the years:  In the
one-dimensional case Feller \cite{feller} characterized in 1957 all
self-adjoint Markovian extensions. Later in 1959 Wentzell
\cite{wentzell} studied Markovian Feller extensions in arbitrary
dimensions, see also \cite{ueno,taira,bony} and references therein.
Under suitable regularity assumptions on the boundary Fukushima
\cite{Fuk}  could  show in 1969
that in arbitrary dimensions all Markovian extensions lie between
the Dirichlet- and the Neumann extension. For special such Markovian
extensions (whose semigroups lie between the Dirichlet- and the
Neumann semigroup) an explicit description in terms of boundary
conditions was then given by Arendt and Warma \cite{ArendtWarma} in
2003. An explicit description of all Markovian extensions in terms
of boundary conditions or rather Dirichlet forms (in the wide sense)
on the boundary
 could only recently be given by  Posilicano \cite{posi}.
 His approach relies on connecting self-adjoint extensions and
boundary conditions via Krein's resolvent formula; a topic with
revived interest in recent years
\cite{Pos,Ryz,BMN,PR,Gr,BGW,Mal,GM}.

In the present paper we are concerned with the analogous questions
for graphs. The study of self-adjoint and Markovian extensions on
infinite weighted graphs has seen an enormous interest in recent
years. On the technical level most of the investigations are phrased
in terms of quadratic forms rather than operators, with Markovian
extensions yielding  Dirichlet form. As for essential
selfadjointness we mention in particular the works
\cite{WojDiss,JP,stoch,HKMW,CdVTHT1,CdVTHT2,TH,Mil}. Of course, in
the essentially selfadjoint case the description of the Markovian
extensions becomes trivial. In the case of general graphs (not
satisfying essential selfadjointness) the investigation of the set
of Markovian extensions becomes meaningful. In \cite{HKLW} a
characterization of the uniqueness of Markovian extensions is  given
and an analogue to the above mentioned result of Fukushima was
shown. More specifically, it was shown that for locally finite
graphs all Markovian extensions of the Laplacian lie between the
Dirchlet and the Neumann form. The condition of local finiteness is
not necessary as can be seen from the considerations of \cite{Sch17}
(which deal with a  much more general framework containing all
Dirichlet forms). However, a description of Markovian extensions in
terms of boundary conditions for graphs is still missing.  The
evident reason for this is certainly that there is no obvious
canonical geometric boundary of an infinite graph.

Now, various works in recent years suggest  that the Royden boundary
of a transient graph can be considered as an analogue to the
geometric boundary of a bounded set in Euclidean space \cite{kasue,
canon, uniform}. Indeed, to develop the theory of transient graphs
according to this point of view can be seen as the main driving
force behind  the considerations in \cite{canon,uniform}. The
present paper is a continuation of this theme building up on the
results  in \cite{canon,uniform} and presenting further confirmation
for this point of view. More specifically, our main result provides
an analogue to the result of Posilicano \cite{posi} mentioned above
and provides a one-to-one correspondence of the Dirichlet forms on a
graph and the Dirichlet forms on the Royden boundary.

 In this context is it worth emphasizing that
our approach is quite different from the one taken in  \cite{posi}.
Indeed, right from the outset our situation is different as there is
no geometrical boundary available. We rather  have  to ``derive'' the
boundary from the form.  Subsequently, our  considerations are then
founded in concepts and methods stemming from forms.  This may well
be of use if it comes to the investigation of corresponding problems
for  general Dirichlet forms.

We subsequently  sketch our results and thereby provide an overview
of the structure of the paper. At the same time this also allows us
to point out  points of contact to existing literature. For now, we
rather aim to convey the basic intuition rather than being precise
but refer to the corresponding sections for the details.

For every graph over a discrete set $X$ there is quadratic form
$\widetilde Q$ on $C(X)$ together with a set $\widetilde D$ -- the
functions of finite energy -- where $\widetilde Q$ stays finite, see
Section~\ref{2.1}. By the use of Gelfand theory the closure of the
algebra $\widetilde D\cap\ell^{\infty}(X)$ is isomorphic to $C(K)$
for some compact set in which $X$ can be densely embedded, see
Section~\ref{2.1}. Then, $\partial X=K\setminus X$ is called the
Royden boundary, Section~\ref{2.6}. Harmonic functions on $X$ induce
harmonic measures on the harmonic boundary $\partial_h X$ (a subset
of $\partial X$) and we choose one such measure $\mu$ as a reference
measure, see Section~\ref{4}.

There are two canonical $\ell^2(X,m)$-restrictions $Q^{(D)}$ and $Q^{(N)}$ of
$\widetilde Q$, where $(D)$ stands for ``Dirichlet-'' and $(N)$
stands for ``Neumann-boundary conditions''. Both forms are Dirichlet
forms.

We consider Dirichlet forms $Q$ that lie between $\QD$ and $\QN$,
i.e., forms which satisfy $Q^{(D)}\geq Q\geq Q^{(N)}$ in the sense
of quadratic forms, see Section~\ref{2.1}, which naturally appear
in the study of Dirichlet forms on discrete spaces, see
Remark~\ref{remark:forms between qd and qn}. For functions  $f$ of
finite energy we define the notion of  a trace $\Tr f$ which plays
the role of boundary values of $f$ on $\partial_h X$, see
Section~\ref{4}. Moreover, the so called Royden decompositions
yields that $f=f_{0}+f_{h}$ with $f_{0}\in D(\QDe)$ and $f_{h}$ is
harmonic, see Section~\ref{2.1} and~\ref{5}. Here the subscript
${\rm e}$ is used to denote the extended Dirichlet form of $\QD$.

The main result of the paper is Theorem~\ref{theorem:main},  which
has two parts. The first part states that $Q$ has the following
decomposition
\begin{align}
Q(f)&=\QDe(f_{0})+q(\Tr f) \label{equation:1}  \\
&= Q^{(N)}(f)+(q-q^{DN})(\Tr f),\label{equation:2} \end{align}
 where $q$ and $q^{DN}$ are Dirichlet forms in the
wide sense on $L^{2}(\partial_h X, \mu)$ and the difference $q-q^{DN}$ is Markovian. The form $q^{DN}$ is the Dirichlet-to-Neumann form relating
$Q^{(D)}$ and $Q^{(N)}$, see Definition~\ref{definition:trace dirichlet form}.

The second part of  Theorem~\ref{theorem:main}  provides a converse
to this decomposition. More precisely, let $q$ be a Dirichlet form
 in the wide sense  on
$L^{2}(\partial_h X,\mu)$ such that $q\geq q^{DN}$ and
$q-q^{DN}$ is Markovian. Then a form $Q$ on $\ell^2(X,m)$ defined  by either of the two previous formulas is a Dirichlet form between $\QD$ and $\QN$. Moreover, if the measure $m$ on the underlying space is finite, these two operations are inverse to each other.

In summary, we achieve a one-to-one correspondence between Dirichlet
forms between $\QD$ and $\QN$ and certain Dirichlet forms on the
harmonic boundary (at least when the underlying measure $m$ is
finite).

We finish the article with a discussion of an example in which
everything can be computed rather explicitly in Section \ref{toy}
and a (counter)example showing that $Q^{(D)}\geq Q$ does not in
general imply $Q \geq Q^{(N)}$ in Section~\ref{appendix:beispiel}.

Along the way
we make use of a few results which are certainly known to the
experts but do not seem to have appeared in print.  For the
convenience of the reader we include a discussion in the appendices.

\smallskip

A few remarks on the history are in order:

\begin{itemize}

\item Induced Dirichlet forms on the (Martin)boundary have been considered
before. In fact, the map $q^{DN}$ has been introduced on such
boundaries  by Silverstein in \cite{Si1}. There, it is also shown that
this form is a pure jump form and calculated the associated jump
measure. The work \cite{Si1} has been somewhat neglected in the past
and has only recently attracted attention by Kong  / Lau / Wong  in
their study of  jump processes on fractals which can be considerd as
Martin boundary of augmented trees \cite{Ka-sing}.

\item The above construction of the trace form can also be derived
as a special case of the theory developed in \cite{CF} based on
stochastic processes. The key virtue of our approach is that it
purely analytic and rather explicit, compare Remark~\ref{remark:our
trace v.s. fukushima trace}.

\item The decomposition $Q(f) = \QDe(f_0) + q(f)$ is known already
and could also in  principle be derived from \cite{CF}. The key
novelty of our approach is the proof of Posilicano's  observation
that $q - q^{DN}$ is Makovian in our context  by form methods (which
is very different from Posilicanos proof). This then allows us to
obtain the new decomposition $Q(f) = \QN(f)  + q'(f)$, which is used
in turn to characterize all Dirichlet forms. In this way, we obtain
an analogue to Posilicano's result in our context by a rather
different proof, see  Remark~\ref{remark:ende} as well.
\end{itemize}

\textbf{Acknowledgments.} M.~S.  acknowledges financial support of
the DFG via \emph{Graduiertenkolleg: Quanten- und
Gravitationsfelder} and M.~K., D.~L. and M.~S. via the Priority Program \emph{Geometry at infinity}.
 The authors gratefully acknowledge enlightening
discussions with Andrea Posilicano on boundary representations of
Dirichlet forms associated with elliptic operators in Euclidean
space.



\section{Transient graphs and the Royden boundary}\label{2.1}
\subsection{Graphs, forms and Laplacians}
In this section we introduce the basic objects of our studies. This includes a discussion of weighted graphs, the associated operators and the associated Dirichlet forms.

Let $X$ be a countably infinite set. A pair $(b,c)$ of functions $b:X\times X\to[0,\infty)$  and $c:X\to[0,\infty)$ is called a {\em weighted graph over $X$} if $b$ is symmetric,
vanishes on the diagonal and satisfies

 \[ \sum_{y\in X} b(x,y)<\infty\] for all $x\in X$.
Elements of $X$ are called {\em vertices} and pairs of vertices $(x,y)$ with $b(x,y) >0$ are called {\em edges}. The function $b$ is an {\em edge weight} and $c$ is a  {\em killing term}.  A finite sequence of vertices $(x_1,\ldots,x_n)$ such that for each $i=1,\ldots,n-1$ the pair $(x_i,x_{i+1})$  is an edge is called a {\em path}. We say that a graph is {\em connected} if for every two vertices $x,y\in X$ there is a path containing $x$ and $y$. From now on we make the  standing assumption that all weighted graphs treated in this paper are connected. For non-connected graphs all of our considerations can be directly applied to their (infinite) connected components.
%
%

Let $C(X)$ be the space of all real-valued functions on $X$ and let $C_c(X)$ be the space of functions in  $C(X)$ with finite support.
For a weighted graph $(b,c)$ over $X$ the domain of the {\em formal Laplacian $\LapUnw$} is defined as
$$\LapDom:=\{f\in C(X):\sum_{y\in X}b(x,y)|f(y)|<\infty \text{ for all } x\in X\},$$
on which it acts pointwise by
$$\LapUnw f(x) :=\sum_{y\in X} b(x,y)(f(x)-f(y))+c(x)f(x).$$
We call a function $h\in\LapDom$ {\em harmonic} if $\LapUnw h \equiv0$. Moreover, the {\em associated energy form} $\Qgen:C(X)\to [0,\infty]$ is given by
$$\Qgen(f)=\frac12\sum_{x,y\in X} b(x,y)(f(x)-f(y))^2+\sum_{x\in X} c(x)f(x)^2$$
and the space of {\em functions of finite energy} is
$$\Dgen:=\{f\in C(X):\Qgen(f)<\infty\}.$$
Polarization gives rise to a semi-scalar product on $\Dgen$, also denoted by $\Qgen$.
In this sense, we have $\ow{Q}(f) = \ow{Q}(f,f)$ whenever $f \in \ow{D}$. We will use the notation $Q(f):=Q(f,f)$ whenever we are dealing with a bilinear form.  According to Fatou's lemma, the form $\Qgen$ is lower semicontinuous with respect to pointwise convergence, i.e., if a sequence $(f_n)$ in $C(X)$ converges pointwise to a function $f \in C(X)$,  then
\[\Qgen(f)\leq\liminf\limits_{n\to\infty}\Qgen(f_n).\]
For a vertex $o \in X$ the connectedness of $(b,c)$ implies that the semi-inner product
$$\as{\cdot,\cdot}_o:\ow{D} \times \ow{D} \to \IR,\, \as{f,g}_o := \ow{Q}(f,g) + f(o) g(o)$$
is indeed an inner product and that the induced topology is independent of the choice of $o$, see e.g. \cite[Lemma~1.3]{uniform}; we denote the corresponding norm by $\|\cdot\|_o$. The following well-known properties of  $(\ow{D},\as{\cdot,\cdot}_o)$ are straightforward from the equivalence of the norms $\|\cdot\|_o$ and $\|\cdot\|_{o'}$, and the lower semicontinuity of $\ow{Q}$. See
\cite[Proposition 3.7]{canon}  for the proof.
\begin{proposition} \label{lemma:properties of dtilde}
 The space $(\ow{D},\as{\cdot,\cdot}_o)$ is a Hilbert space. Moreover, any sequence in $\ow{D}$ that converges with respect to $\|\cdot\|_o$ also converges pointwise and the corresponding limits agree.
\end{proposition}
We let $$\Dzero:=\overline{C_c(X)}^{\|\cdot\|_o}$$ be the closure of $C_c(X)$ in $(\ow{D},\as{\cdot,\cdot}_o)$. It is independent of the choice of $o \in X$ and can be thought of as the space of functions of finite energy that vanish at infinity in a suitable sense.

Recall that a function $C:\IR\to\IR$ is called {\em normal contraction} if both $|C(r)|\leq |r|$ and $|C(r)-C(s)|\leq|r-s|$ hold  for every $r,s\in\IR$.
A quadratic form $Q$ whose domain $ D(Q) $ is a vector space of functions
is called \emph{Markovian} if $C\circ f\in D(Q)$ and $Q(C\circ f)\leq Q(f)$ hold for every normal contraction $C$ and every $f\in D(Q)$.

Obviously, $\Qgen$ is Markovian, since $\Qgen(C\circ f)\leq \Qgen(f)$ holds for every $f\in C(X)$ and every normal
contraction $C$ by a simple calculation and, hence, $f \in \ow{D}$ implies $C \circ f \in \ow{D}$. This   statement  is also true
with $\ow{D}$ being replaced by $\Dzero$, i.e., for any normal contraction $C$ and any $f \in \Dzero$ we have $C \circ f \in \Dzero$, see e.g. \cite[Lemma~1.5]{uniform}.

The connection between $\Qgen$ and $\LapUnw$ is the following  discrete version of Green's formula. In the presented form it is a slight generalization of \cite[Lemma~7.7]{kasue}. We include a proof for the convenience of the reader.

\begin{proposition}[Green's formula]\label{Green}
 Let $(b,c)$ be a graph over $X$. Then $\ow{D} \subseteq \ow{F}$ and for $f \in \Dzero$ and $g \in \ow{D}$ that satisfy
 $$\sum_{x \in X} |f(x) \mathcal{L} g (x)| < \infty$$
 we have
 $$\Qgen(f,g)=\sum_{x\in X}f(x)\LapUnw g(x).$$
 If  $f \in C_c(X)$, then
 $$\ow{Q}(f,g) =  \sum_{x\in X}\LapUnw f(x) g(x).$$
\end{proposition}
\begin{proof}
 According to \cite[Proposition~3.8]{HKLW}, the inclusion  $\ow{D} \subseteq \ow{F}$ holds. For $f \in C_c(X)$ and $g \in \ow{D}$ \cite[Lemma~4.7]{solu} yields
$$\Qgen(f,g)=\sum_{x\in X}f(x)\LapUnw g(x) = \sum_{x\in X}\LapUnw f(x) g(x).$$
It remains to prove that the first equality remains true for $f \in \Dzero$.

To this end, let $f \in \Dzero, g\in \ow{D}$ and choose a sequence $(\varphi_n)$ in $C_c(X)$ that converges to $f$ with respect to $\|\cdot\|_o$. According to Proposition~\ref{lemma:properties of dtilde}, the sequence $(\varphi_n)$ converges pointwise to $f$. We set
%
$$f_n :=  ((\varphi_n \wedge |f|)\vee 0)-(((-\varphi_n) \wedge |f|)\vee 0). $$
The sequence $(f_n)$ belongs to $C_c(X)$, converges pointwise to $f$ and satisfies $|f_n| \leq |f|$. Using that $\ow{Q}^{\frac12}$ is Markovian and a semi-norm, we infer
$$\ow{Q}(f_n)^{1/2}\leq\ow{Q}(\varphi_n \wedge |f|)^{\frac12}+\ow{Q}((-\varphi_n) \wedge |f|)^{\frac12}.$$  For $a,b \in \IR$ the identity $2 (a \wedge b) = a + b - |a-b|$ holds. Therefore, the compatibility of $\ow{Q}$ with normal contractions implies
$$\ow{Q}(f_n)^{1/2} \leq 2\ow{Q}(\varphi_n)^{1/2} + 2 \ow{Q}(f)^{1/2}.$$
This shows that $(f_n)$ is also $\ow{Q}$-bounded and so we can infer $f_n \to f$ $\ow{Q}$-weakly from Lemma~\ref{lemma:q-wekly convergent sequences}. Hence, Lebesgue's dominated convergence theorem and what was already proven yields
$$\ow{Q}(f,g) = \lim_{n \to \infty} \ow{Q}(f_n,g) = \sum_{x \in X} f_n(x) \mathcal{L} g(x) =\sum_{x \in X} f(x) \mathcal{L} g(x).$$
This finishes the proof.
\end{proof}

Next we discuss Dirichlet forms on graphs.

A  function $m:X\to [0,\infty)$ gives rise to a measure on  subsets of $X$, also denoted by $m$, via
\[m(A):=\sum_{x\in A} m(x).\]
We call the pair $(X,m)$ a {\em discrete measure space} if $m$ has full support, i.e., if $m(x)>0$ holds for every $x\in X$. In this case, we let
$\ltwo$ be the space of $m$-square summable real-valued functions on $X$ with the corresponding  norm $\|\cdot\|_2$
and inner product $\langle\cdot,\cdot\rangle_2$.

For a graph $(b,c)$ over a discrete measure space $(X,m)$, we let $\NeuDir$ be the restriction of $\ow{Q}$ to $D(\NeuDir):= \ow{D} \cap \ell^2(X,m)$ and we let $\DirDir$ be the restriction of $\ow{Q}$ to
$$D(\DirDir) := \overline{C_c(X)}^{\|\cdot\|_\NeuDir},$$
where the closure is taken in $D(\NeuDir)$ with respect to the form norm
$\|\cdot\|_\NeuDir$ on $D(\NeuDir)$ given by $\|f\|_\NeuDir:= (\NeuDir(f)+\|f\|^2_2)^{\frac 12}. $
Both $\DirDir$ and $\NeuDir$ are Dirichlet forms, see e.g. \cite{HKLW}. The form $\DirDir$ can be thought of as having Dirichlet boundary conditions at infinity and $\NeuDir$  can be thought of as having Neumann boundary conditions at infinity.  If necessary we highlight the dependence of $\QD$ and $\QN$ on the graph by adding the subscript $(b,c)$, i.e., we write $Q^{(D)}_{(b,c)}$ or $Q^{(N)}_{(b,c)}$, respectively.

Given a quadratic form $Q$ on a vector space $V$ with domain $D(Q) \subseteq V$, we set $Q(f) := \infty$ whenever $f \in V \setminus D(Q)$. In this sense, $D(Q) = \{f \in V\, :\, Q(f) < \infty\}$. We say that two quadratic forms $Q_1,Q_2$ on $V$ satisfy $$Q_1 \geq Q_2$$ if for all $f \in V$ we have $Q_1(f) \geq Q_2(f)$. This is equivalent to  $D(Q_1) \subseteq D(Q_2)$ and $Q_1(f) \geq Q_2(f)$ for all $f \in D(Q_1)$. As mentioned in the introduction, for a weighted graph $(b,c)$ over a discrete measure space $(X,m)$ it is the main goal of this paper to characterize all Dirichlet forms $Q$ on $\ell^2(X,m)$ that satisfy $\QD \geq Q \geq \QN$.

\begin{remarks} \label{remark:forms between qd and qn}
 On a discrete space it is quite natural to consider Dirichlet forms that satisfy $\QD\geq Q \geq \QN$.
 If $Q$ is a Dirichlet form on $\ell^2(X,m)$ such that $D(Q) \cap C_c(X)$ is dense in $C_c(X)$ with respect to uniform convergence, it follows from the considerations in \cite[Section~2]{stoch} that there exists a graph $(b,c)$ over $X$ such that $Q$ extends the  form $Q^{(D)}_{(b,c)}$ and so  $Q^{(D)}_{(b,c)} \geq Q$.
 If $c = 0$,  the discussion in \cite[Chapter~3]{Sch} implies that  any  Dirichlet form $Q$ that extends $\DirDir$ automatically satisfies $Q \geq \NeuDir$. Indeed, we shall perform the computations to prove this result in Section~\ref{section:differences of dirichlet forms} for other purposes.
 If $c \neq 0$ and $Q$ is an extension of $\DirDir$, it need not be true that   $Q \geq \NeuDir$ without any additional assumptions. In the more general context of Silverstein extensions of energy forms this phenomenon was discovered in \cite{Sch}; we give an example that fits into our setting of Dirichlet forms on infinite graphs in Section~\ref{appendix:beispiel}.   One way of guaranteeing that  $Q \geq \NeuDir$  is demanding that the self-adjoint operator associated with $Q$ is a restriction of the scaled formal Laplacian $\frac{1}{m} \mathcal{L}$. For locally finite graphs this is discussed in \cite{HKLW}.
\end{remarks}
Let  $Q$ be a Dirichlet form on $\ell^2(X,m)$. We say that a sequence $(f_n)$ in $D(Q)$ is $Q$-Cauchy if it is a Cauchy-sequence with respect to the semi-norm $Q(\cdot)^{\frac12}$. The {\em extended Dirichlet space $D(\Qe)$} of $Q$ is defined by
 $$D(\Qe) :=  \{f \in C(X)\, : \, \text{ex. }  Q \text{-Cauchy sequence } (f_n) \text{ in } D(Q) \text{ with } f_n \to f \text{ pointwise}\}.$$
 Let $f \in D(\Qe)$ and let $(f_n)$ be  a $Q$-Cauchy sequence in $D(Q)$ with $f_n \to f$ pointwise; such a sequence is called {\em approximating sequence for $f$}. We  extend $Q$ to $D(\Qe)$ by letting
 $$\Qe(f) := \begin{cases}
              \lim_{n \to \infty} Q(f_n) &\text{if } f\in D(\Qe) \text{ and } (f_n) \text{ is an approximating sequence for }f,\\
              \infty &\text{else}.
             \end{cases}
$$
It is proven in  \cite[Lemma 1]{Schmu} that the definition of $\Qe$ does not depend on the choice of the approximating  sequence. Moreover, $\Qe$ is a quadratic form on $C(X)$, the so-called {\em extended Dirichlet form of $Q$}. For later purposes we need the following result on lower semicontinuity of the extended form with respect to pointwise convergence. It is a special case of \cite[Lemma~3]{Schmu2}, see also \cite[Theorem~1.59]{Sch} for a simplified proof.
\begin{proposition}  \label{proposition:pointwise lower semicontinuity extended space}
 Let $Q$ be a Dirichlet form on $\ell^2(X,m)$ and let $\Qe$ be its extended Dirichlet form. For every sequence $(f_n)$ in $C(X)$ that converges pointwise to some $f \in C(X)$
 $$\Qe(f) \leq \liminf_{n\to \infty}\Qe(f_n).$$
\end{proposition}
The previous proposition implies  that $\Qe$ inherits the Markov property from $Q$. We finish this subsection  by mentioning the following properties of the extended forms
of Dirichlet forms between $\QD$ and $\QN$.
\begin{lemma} \label{lemma:extended dirichlet spaces of qd and qn}
 Let $(b,c)$ be a graph over a discrete measure space $(X,m)$. Let $Q$ be a Dirichlet form on $\ltwo$ such that $\DirDir\geq Q\geq \NeuDir$. Then $\QDe \geq \Qe \geq \QNe$ and
 $$\Dzero = D(\QDe) \subseteq D(\Qe) \subseteq D(\QNe) \subseteq \ow{D}.$$
 Moreover, for every $f \in \Dzero$ we have
 $$\QDe(f)= \Qe(f) = \QNe(f)  = \ow{Q}(f),$$
 and for every $f \in D(\QNe)$ we have
 $$\QNe(f) = \ow{Q}(f).$$
\end{lemma}
\begin{proof}
 The definition of the extended Dirichlet space implies $\QDe \geq \Qe \geq \QNe$ and the definition of $\QD$ and $\QN$ and the inequalities $\QD \geq Q \geq \QN$ yield
$$\QD(f) = \QN(f) = Q(f) = \ow{Q}(f)\text{ for all } f \in D(\QD).$$
 From this equality it immediately follows that
$$\QDe(f) = \QNe(f) = \Qe(f) \text{ for all } f \in D(\QDe). $$
 The equality $D(\QDe) = \Dzero$  is a consequence of Proposition~\ref{lemma:properties of dtilde}. To prove the claim on $\ow{Q}$, for $f \in \Dzero$ we let $(f_n)$ be a $\QD$-Cauchy sequence in $D(\QD)$ that converges pointwise to $f \in C(X)$. The lower semicontinuity of $\ow{Q}$ shows
$$\ow{Q}(f - f_n) \leq \liminf_{m \to \infty} \ow{Q}(f_m-f_n) = \liminf_{m \to \infty} \QD(f_m-f_n), $$
and so we obtain $\ow{Q}(f) = \QDe(f)$. A similar computation with $\QD$ being replaced by $\QN$ shows $D(\QNe) \subseteq \ow{D}$ and $\ow{Q}(f) = \QNe(f)$ for $f \in D(\QNe)$. This finishes the proof.
\end{proof}

We  note the following properties of $\ow{D}$, the domains of Dirichlet forms and extended Dirichlet forms.
\begin{proposition}\label{proposition:algebraic and order properties of domains}
 Let $(b,c)$ be a graph over the discrete measure space $(X,m)$. Then $\ow{D} \cap \ell^\infty(X)$ is an algebra and $\ow{D}$ is a lattice, i.e. for every $f,g\in\ow{D}$ we have
 $f\vee g,f\wedge g\in \ow{D}$. Moreover, for any $f \in \ow{D}$ the convergence
 $$\ow{Q}(f- f^{(n)}) \to 0,\text{ as } n \to \infty,$$
holds, where $f^{(n)} := (f \wedge n) \vee (-n)$. The same is true with $(\ow{Q},\ow{D})$ being replaced by $(Q,D(Q))$ or $(\Qe,D(\Qe))$, where  $Q$ is a Dirichlet form on $\ell^2(X,m)$.
\end{proposition}
\begin{proof}
By straightforward computations the assertions for $\ow{D}$ and $\ow{Q}$ follow. For the corresponding statements for Dirichlet forms see  \cite[Theorem~1.4.2, Corollary~1.5.1]{FOT}.
\end{proof}

\subsection{Recurrence and transience}
In this section we discuss transient graphs and list some basic -- but important -- facts. For our purposes it is convenient to work with the following definition of recurrence and transience. For a detailed discussion on how this is related to other (analytic and probabilistic) notions of recurrence and transience we refer the reader to \cite{Sch17}.

A weighted graph $(b,c)$ is called {\em recurrent} if $1\in\Dzero$ and the equality $\widetilde{Q}(1)=0$ holds. It
is called {\em transient} if for all $f \in \Dzero$ the equality $\ow{Q}(f) = 0$ implies $f = 0$. A characterization of transience can be found in Appendix \ref{section:Char_Trans}.

\begin{remarks}
 Clearly,  every recurrent graph satisfies $c\equiv 0$. Moreover, the connectedness  of $(b,c)$ implies the dichotomy of recurrence and transience.
\end{remarks}

The next lemma shows that $\Dzero$ and the space of energy finite harmonic functions are orthogonal with respect to $\widetilde{Q}$. It follows directly from
the Green's formula in Proposition~\ref{Green}, see also \cite[Lemma~3.66]{soardi}.
\begin{lemma}\label{orthharm}
Let $(b,c)$ be a graph over $X$ and let $h \in \ow{D}$.  The following assertions are equivalent.
\begin{itemize}
\item[(i)] $h$ is harmonic.
 \item [(ii)] $\widetilde{Q}(h,f)=0$ for all $f\in\Dzero$.
\item[(iii)] $\widetilde{Q}(h,f)=0$ for all $f \in C_c(X)$.
\end{itemize}
\end{lemma}
The following decomposition of functions of finite energy is one of the main tools in this paper. For graphs without killing it is the content of \cite[Theorem 3.69]{soardi} and for graphs with killing we refer to \cite[Proposition~5.1]{uniform}.
\begin{theorem}[Royden decomposition]\label{theorem:royden decomposition}
 Let $(b,c)$ be a transient graph. For all $f\in\Dgen$ there exists a unique $f_0\in\Dzero$ and a unique harmonic $f_h\in\Dgen$ such that
 $$f=f_0+f_h$$
 $$\Qgen(f)=\Qgen(f_0)+\Qgen(f_h).$$
  Moreover, if $f$ is bounded, then $f_0$ and $f_h$ are bounded as well.
\end{theorem}
%
%
For a function $f \in \ow{D}$ we call the pair $(f_0,f_h)$ of the previous theorem its {\em Royden decomposition}. In particular, we use the subscripts $0$ and $h$ on functions for the corresponding parts in the Royden decomposition. For later purposes we note the following continuity property of the Royden decomposition.

\begin{lemma} \label{lemma:poinwise continuity royden decomposition}
Let $(b,c)$ be a transient graph and let $o \in X$. Let $(f_n)$ be a sequence in $\ow{D}$ that converges to some $f \in \ow{D}$ with respect to $\|\cdot\|_o$. Then $(f_n)_0 \to f_0$ and $(f_n)_h \to f_h$ pointwise, as $n \to \infty.$
\end{lemma}
\begin{proof}
Convergence with respect to $\|\cdot\|_o$ implies pointwise convergence, see Proposition~\ref{lemma:properties of dtilde}, and so it suffices to prove $(f_n)_0 \to f_0$ pointwise. This however is a consequence of the characterization of transience, see Theorem~\ref{CharTrans}.
\end{proof}


\subsection{The Royden boundary}\label{2.6}
Next we recall  the  Royden compactification and the Royden boundary of a weighted graph. For more details see e.g. \cite{soardi} and the extended\footnote{arXiv:1309.3501} version of \cite{canon}.

 Let $(b,c)$ be a graph over the set $X$. The uniform closure of $\BFEspace$ in $\linf$ is denoted by $$\A:=\overline{\ow{D}\cap\ell^\infty(X)}^{\|\cdot\|_\infty}$$ and
 is called {\em Royden algebra.} The algebra generated by $\A$ and the constant function $1$ is denoted by $\A^+$.

By applying Gelfand theory  to (the complexification of) $\A^+$ we infer the existence of a unique (up to homeomorphism) separable, compact Hausdorff space $R$
 such that $\A^+$ is isomorphic to $C(R)$, the algebra of real-valued continuous functions on $R$. Using standard arguments, it follows  that the following conditions are satisfied:
\begin{itemize}
 \item $X$ can be embedded into $R$ as a dense open subset,
 \item every function in $\BFEspace$ can be uniquely extended to a continuous function on $R$,
 \item the algebra $\BFEspace$ separates the points of $R$.
\end{itemize}
The set $R$ is called the {\em Royden compactification} of the graph $(b,c)$ over $X$.

 In what follows we tacitly identify functions in $\A^+$ with continuous function on $R$.

 Let $(b,c)$ be a graph over the set $X$ and let $R$ be its Royden compactification. The set $$\partial X:=R\setminus X$$ is called the {\em Royden boundary} of $(b,c)$ and
 $$\partial_h X:=\{z \in\partial X\, : \,  f(z)=0 \text{ for all } f\in \Dzero\cap \linf\}$$
is called the {\em harmonic boundary} of $(b,c)$.

%
%
%
Since $X$ is open in $R$, the Royden boundary $\partial X = R \setminus X$ is compact. As an intersection of closed sets in the compact set $\partial X$, the harmonic boundary is  compact as well.
%
%
%
\begin{remarks}
 \begin{itemize}
 \item[(a)] According to \cite[Proposition~4.9]{canon},  $1\in\A$  holds if and only if $\sum_{x\in X} c(x)<\infty$. In this case, $\A=\A^+$.
 \item[(b)] The Royden boundary of an infinite graph is always nonempty. For the harmonic boundary, however, this is not always the case. In \cite[Proposition~5.4]{uniform} it is proven that
 $\partial_h X=\emptyset$ holds if and only if $1\in\Dzero$, which is equivalent to $\sum_{x\in X} c(x)<\infty$ and $(b,0)$ being recurrent.
 \item[(c)] Even if $(b,0)$ is transient, it may happen that $\partial_h X \neq \partial X$.  Graphs for which $\partial_h X = \partial X$ holds are called uniformly transient, see \cite{uniform}.
 \end{itemize}
 \end{remarks}

The importance of the harmonic boundary stems from the following maximum principle for harmonic functions in $\BFEspace$, see e.g. \cite[Corollary~5.3]{uniform}.
\begin{proposition}[Maximum principle]\label{maxprin}
  Let  $(b,c)$ be a transient graph over $X$ with $\partial_h X\not=\emptyset$ and let $h\in\BFEspace$ be harmonic.
 Then,
 $$\sup_{x \in X} |h(x)| = \sup_{z \in \partial_h X} | h (z)|.$$

\end{proposition}
One important consequence of the maximum principle is that functions in $\Dzero \cap \ell^\infty(X)$ are exactly those that vanish on $\partial_h X$. This observation is an extension of \cite[Corollary~6.8]{soardi} to graphs which possibly satisfy $c \neq 0$.
\begin{corollary} \label{corollary:d0 as kernel of gamma0}
 Let  $(b,c)$ be a transient graph over $X$ such that $\partial_h X\not=\emptyset$. Then
 $$\Dzero \cap \ell^\infty(X) = \{f \in \ow{D} \cap \ell^\infty(X) \,:\, f(z) = 0 \text{ for all } z \in \partial_h X\}.$$
\end{corollary}
\begin{proof}
The definition of $\partial_h X$ implies
$$\Dzero \cap \ell^\infty(X) \subseteq \{f \in \ow{D} \cap \ell^\infty(X) \,:\, f(z) = 0 \text{ for all } z \in \partial_h X\}.$$
For proving the opposite inclusion  we let $f \in \ow{D} \cap \ell^\infty(X)$ with $ f(z) = 0 \text{ for all } z \in \partial_h X$. The Royden decomposition theorem, Theorem~\ref{theorem:royden decomposition}, yields  $f = f_0  + f_h$ with $f_0 \in \Dzero \cap \ell^\infty(X)$ and harmonic  $f_h \in \ow{D} \cap \ell^\infty(X)$. Since functions in $\Dzero \cap \ell^\infty(X)$ vanish on the harmonic boundary, we obtain $f_h(z) = f(z) = 0$ for all $z \in \partial_h X$. By the maximum principle this implies $f_h = 0$ and so   $f = f_0 \in \Dzero \cap \ell^\infty(X)$.
\end{proof}

As subset of $R$ the harmonic boundary $\partial_h X$ is a topological space and we denote by $C(\partial_h X)$  the continuous real-valued functions on $\partial_h X$. We define the {\em preliminary trace map}
$$\gamma_0 :\mathcal{A} \to C(\partial_h X),\, f \mapsto f|_{\partial_h X}.$$
In the next section we introduce certain measures $\mu$ on $\partial_h X$ and (partially) extend $\gamma_0$ to a map $\Tr : \ow{D} \to L^2(\partial_h X,\mu)$. The image of $\gamma_0$ can be expressed in the following way, see \cite[Lemma~4.8]{uniform}.
\begin{lemma}\label{Adarst}
 Let $(b,c)$ be a transient graph over $X$ with $\partial_h X\not=\emptyset$.
 \begin{itemize}
  \item [(a)] If $1\in\A$, then $\gamma_0 \mathcal{A} =C(\partial_h X)$.
  \item [(b)] If $1\not\in\A$, then there exists a point $p_\infty \in \partial_h X$ such that
  $$\gamma_0 \mathcal{A}=\{f\in C(\partial_h X): f(p_\infty)=0\}.$$
 \end{itemize}
 \end{lemma}
 The existence result in the next theorem is taken from \cite[Theorem~5.5]{uniform}. The uniqueness for functions in $\gamma_0(\ow{D}\cap\ell^\infty(X))$ follows from the
 maximum principle.
\begin{theorem}[Dirichlet problem] \label{Dirprob}
 Let $(b,c)$ be a transient graph over $X$ with $\partial_h X\not=\emptyset$. For every $\varphi\in \gamma_0 \mathcal{A}$ the equation
 $$\begin{cases}\mathcal{L} h =0 \\ \gamma_0 h=\varphi \end{cases}$$
 has a  solution $h_\varphi\in\A$ and has at most one solution in $\BFEspace$. In particular, for all $\varphi \in \gamma_0(\ow{D} \cap \ell^\infty(X))$ the above equation has a unique solution in $\ow{D} \cap \ell^\infty(X)$.
%
%
\end{theorem}
\begin{proof}
 The existence of a solution $h_\varphi \in \mathcal{A}$ is contained in  \cite[Theorem 5.5]{uniform} and the uniqueness of solutions in $\ow{D}\cap \ell^\infty(X)$
 follows from the maximum principle, Proposition~\ref{maxprin}. The ``In particular''-part follows from the Royden decomposition and the definition of the harmonic boundary.
\end{proof}

\section{Harmonic measures and the trace map}\label{4}
\subsection{Harmonic measures}
In this section we introduce   measures on the harmonic boundary that are determined by   harmonic functions and employ them to extend the map $\gamma_0$ to all functions of finite energy.  The following theorem is a variant of \cite[Theorem~6.40 and Theorem~6.43]{soardi}.
\begin{proposition}[Existence of harmonic measures]\label{harmmeas}
Let $(b,c)$ be a transient graph over $X$ with $\partial_h X\not=\emptyset$. For every $x \in X$ there exists a regular Borel measure $\mu_x$ on $\partial_h X$ with $\mu_x(\partial_h X) \leq 1$ such that for every harmonic $h \in \ow{D} \cap \ell^\infty(X)$ we have
$$  h(x) = \int_{\partial_h X} \gamma_0 h \, d\mu_x.$$
If $1 \in \A$, then $\mu_x$ is uniquely determined and if $1 \not\in \A$, then $\mu_x$ is uniquely determined up to the value of $\mu_x(\{p_\infty\})$, where $p_\infty$ is the point  mentioned in Lemma~\ref{Adarst}. Moreover, for all $x,y \in X$ there exists a bounded Borel measurable function $K_{x,y}:\partial_h X \to [0,\infty)$ such that $\mu_x = K_{x,y} \cdot \mu_y$.
%
%
%
\end{proposition}
\begin{proof}
First suppose that $1$ belongs to the algebra $\mathcal{A}$. Let $z\in X$ be arbitrary. Then, the idea of the existence of $\mu_z$ is as follows: every function $\varphi$ in $\gamma_0(\ow{D} \cap \ell^\infty(X))$ can be uniquely extended to a harmonic function $\widetilde{\varphi}$ in $\ow{D} \cap \ell^\infty(X)$ by Theorem~\ref{Dirprob}. This mapping is linear and positivity preserving (as a consequence of Theorem~\ref{theorem:royden decomposition}). Moreover, by Lemma~\ref{Adarst}, the set
  $\gamma_0(\ow{D} \cap \ell^\infty(X))$ is uniformly dense in $C(\partial_h X)$. Therefore, the mapping $\varphi\mapsto \widetilde{\varphi}(z)$ is a positive and bounded (by Proposition~\ref{maxprin}) linear functional on a dense subspace, and, using Riesz-Markov, we infer the existence of $\mu_z$. For a detailed proof of the existence and the properties of $\mu_z$, we
refer to \cite[Theorem~6.40]{soardi}. Note, that this theorem treats the case $c \equiv 0$. However, the proof carries over to our situation. The existence of the functions $K_{x,y}$ can be proven as \cite[Theorem~6.43]{soardi}.

Next suppose $1 \not\in \mathcal{A}$. Then, the proofs of  \cite[Theorem~6.40 and Theorem~6.43]{soardi} can be employed to construct Borel measures $\mu_x$ on $\partial_h X \setminus \{ p_\infty\}$ and bounded measurable functions $K_{x,y}:\partial_h X \setminus\{p_\infty\} \to [0,\infty)$ such that the theorem holds with $\partial_h X$ being replaced by  $\partial_h X \setminus \{p_\infty\}$, where $p_\infty$ is the point  mentioned in Lemma~\ref{Adarst}. We extend the $\mu_x$ and $K_{x,y}$ to $\partial_h X$ by letting $\mu_x(\{p_\infty\}) = 0$ and $K_{x,y}(p_\infty) = 0$. Clearly these extensions possess the desired properties. The uniqueness statement follows from the fact that $\gamma_{0} (\ow{D}\cap\ell^\infty(X))$ is uniformly dense in $\{f\in C(\partial_h X): f(p_\infty)=0\}$ by Lemma~\ref{Adarst}.
\end{proof}
%
%
 Let $(b,c)$ be a transient graph over $X$ with $\partial_h X\not=\emptyset$. For every $x\in X$ the unique finite regular Borel measures $\mu_x$ on $\partial_h X$ that
 satisfy the assertions of Proposition \ref{harmmeas} (with $\mu_x(\{p_\infty\}) = 0$ if $1 \not \in \A$)  is called {\em harmonic measure at point $x$}.
 The corresponding family of bounded Borel measurable functions $K_{x,y}$, $x,y \in X$, with $\mu_x = K_{x,y} \cdot \mu_y$ is called the {\em harmonic kernel}.

In what follows we call a measure $\mu$ on $\partial_h X$ harmonic if there is $z \in X$ such that $\mu = \mu_z$ holds.

Let $\mu$ be a harmonic measure. The boundedness of the harmonic kernels in particular guarantees
the equality $$L^p(\partial_h X, \mu)=L^p(\partial_h X, \mu_x)$$ for every $x\in X$ and $p\geq 1$.

 Let $(b,c)$ be a transient graph over $X$ with $\partial_h X\not=\emptyset$ and let $\mu$ be a harmonic measure. For $\varphi \in L^1(\partial_h X, \mu)$ the {\em harmonic extension} $H_\varphi:X \to \IR$ of $\varphi$ is defined pointwise by
$$H_\varphi (x) := \int_{\partial_h X} \varphi   \,d\mu_x.$$

   The boundedness of the harmonic kernel yields that $H_\varphi$  exists for all $\varphi \in L^1(\partial_h X, \mu)$ and implies the following continuity statement.
\begin{lemma}\label{lemma:pointwise continuity harmonic extension}
 Let $(b,c)$ be a transient graph over $X$ with $\partial_h X \neq \emptyset$ and let $\mu$ be a harmonic measure. If $\varphi_n \to \varphi$ in $L^1(\partial_h X,\mu)$, then $H_{\varphi_n} \to H_\varphi$ pointwise.
\end{lemma}
\begin{proof}
Let $x\in X$ be arbitrary and $z\in X$ be such that $\mu=\mu_z$. Then, we infer
\[|H_{\varphi_n}(x)-H_\varphi(x)|\leq \int_{\partial_h X} |\varphi_n-\varphi|d\mu_x=\int_{\partial_h X} |\varphi_n-\varphi|K_{x,z}d\mu.\]
Thus, the boundedness of $K_{x,z}$ and the $L^1(\partial_h X, \mu)$-convergence of $(\varphi_n)$ yield the result.
\end{proof}
It can be proven that for any $\varphi \in L^1(\partial_h X,\mu)$ its harmonic extension $H_\varphi$ is indeed a harmonic function. We restrict ourselves to the following special case. Recall that we use the subscript $h$ to denote the harmonic part of a function in its Royden decomposition.
\begin{proposition}[Properties of solutions to the Dirichlet problem]\label{Hgamma}
 Let $(b,c)$ be a transient graph over $X$ with $\partial_h X\not=\emptyset$. For every $\varphi\in \gamma_0 (\BFEspace)$ the function $H_\varphi$
 is the unique solution of the Dirichlet problem
 $$\begin{cases}\mathcal{L} h = 0, \\ \gamma_0 h = \varphi,\quad h \in \ow{D}.\end{cases}$$
In particular, for every $f\in\Dgen \cap \ell^\infty(X)$ we have
$$H_{\gamma_0 f}=f_h$$
 and for every normal contraction $C:\IR \to \IR$ and every $\varphi\in \gamma_0 (\Dgen \cap \ell^\infty(X))$ we have
$$H_{C\circ \varphi}=(C\circ (H_\varphi))_h.$$
\end{proposition}
\begin{proof}
  Let $\varphi\in \gamma_0(\BFEspace)$. According to Theorem~\ref{Dirprob}, there is a unique harmonic $h\in\BFEspace$ such that
 $\gamma_0 h =\varphi$. Moreover, Proposition \ref{harmmeas} yields
 $$h(x)=\int_{\partial_h X} \varphi  \,d\mu_x$$
 for every $x\in X$. From the definition of $H_\varphi$ we infer $H_\varphi=h$, and so $H_\varphi$ is the unique solution to the Dirichlet problem.

Let $f \in \ow{D}\cap \ell^\infty(X)$ and let $f = f_0 + f_h$ be its Royden decomposition. According to Theorem~\ref{theorem:royden decomposition},
we have $f_0 \in \Dzero \cap \ell^\infty(X)$ and so $\gamma_0 f_0 = 0$ on $\partial_h X$ by the definition of the harmonic boundary.
This implies $\gamma_0 f = \gamma_0 f_h$. Therefore, $f_h \in \ow{D}$ is a solution to the Dirichlet problem with boundary value $\gamma_0f$.
The uniqueness proven above yields $f_h = H_{\gamma_0 f}$.

 Let $C:\IR \to \IR$ be a normal contraction and let $\varphi \in \gamma_0(\ow{D} \cap \ell^\infty(X))$. Since $\gamma_0$ operates by restricting a function on $R$ to a function on $\partial_h X$ and since it only depends on the harmonic part, we obtain
 $$\gamma_0 ((C\circ H_\varphi)_h)=\gamma_0 (C\circ H_\varphi)=C\circ (\gamma_0 H_\varphi )= C\circ \varphi.$$
 For the last equality we used that $H_\varphi$ solves the Dirichlet problem. With this at hand, the desired equality follows from the identity $H_{\gamma_0 ((C\circ H_\varphi)_h)} = (C\circ H_\varphi)_h$,  which was already proven. This finishes the proof.
\end{proof}

\subsection{The trace map}
The following theorem shows that the restriction of the preliminary trace map $\gamma_0$ to $\ow{D} \cap \ell^\infty(X)$ is a continuous map from $(\ow{D}\cap\ell^\infty(X),\|\cdot\|_o)$
to $L^2(\partial_h X,\mu)$. It is proven in \cite[Lemma 7.8]{kasue} for the Kuramochi boundary but the proof given there works for the Royden boundary as well.
We include a proof in Appendix \ref{appendix: proof kasue} for the convenience of the reader.

\begin{theorem}[Continuity of the preliminary trace map]\label{gammaLtwo}
Let $(b,c)$ be a transient graph over $X$ with $\partial_h X\not=\emptyset$ and let $\mu$ be a harmonic measure.  For every $o\in X$ there exists a constant $C_o \geq 0$ such that for all $f \in \ow{D}\cap \ell^\infty(X)$ we have
$$\int_{\partial_h X} (\gamma_0 f)^2 \, d\mu \leq C_o  \|f\|_o^2.$$
\end{theorem}
 We note several important consequences of the preceding theorem.
\begin{corollary}[The trace map]\label{corollary: properties of trace map}
 Let $(b,c)$ be a transient graph with $\partial_h X \neq \emptyset$, let $\mu$ be a harmonic measure and let $o \in X$. There exists a unique continuous linear operator
 $\Tr:(\ow{D},\|\cdot\|_o) \to L^2(\partial_h X, \mu)$
 that extends $\gamma_0:\ow{D} \cap \ell^\infty(X) \to C(\partial_h X)$.
\end{corollary}
\begin{proof}
 By the previous theorem the operator $\gamma_0:(\ow{D} \cap \ell^\infty(X),\|\cdot\|_o) \to L^2(\partial_h X,\mu)$ is continuous and  $\ow{D}\cap \ell^\infty(X)$ is dense in
 $\ow{D}$ with respect to $\|\cdot\|_o$ by Proposition~\ref{proposition:algebraic and order properties of domains}. Therefore, $\gamma_o$ has a unique continuous extension $\Tr:(\ow{D},\|\cdot\|_o) \to L^2(\partial_h X, \mu)$.
\end{proof}
\begin{definition}
Let $(b,c)$ be a transient graph with $\partial_h X \neq \emptyset$, let $\mu$ be a harmonic measure and let $o \in X$.
We call the operator $$\Tr:(\ow{D},\|\cdot\|_o) \to L^2(\partial_h X, \mu)$$ the \emph{trace map}.
\end{definition}

\begin{remarks}
 For different $o \in X$ the norms $\|\cdot\|_o$ are equivalent and the boundedness of the harmonic kernel ensures that all $L^2$-spaces with respect to harmonic measures
 are equal. Hence, the trace map is independent of the choice of $o$  and the harmonic measure.
\end{remarks}

The following lemma shows that the properties of $\gamma_0$ extend to $\Tr$.
\begin{lemma}[Properties of the trace map]\label{proposition:properties of trace}
  Let $(b,c)$ be a transient graph with $\partial_h X \neq \emptyset$.
  \begin{itemize}
   \item[(a)] The trace map commutes with normal contractions, i.e., for all normal contractions $C:\IR \to \IR$ and all $f \in \ow{D}$ we have $\Tr (C\circ f) = C \circ (\Tr f)$.
   \item[(b)] The kernel of the trace map satisfies $\ker \Tr = \Dzero.$
   \item[(c)]  For any $f \in \ow{D}$ we have $H_{\Tr f} = f_h$.
   \item[(d)] For any $\varphi \in \Tr \ow{D}$ we have $\varphi = \Tr H_\varphi.$ Moreover, if $C:\IR \to \IR$ is a normal contraction, then
   $$H_{C \circ \varphi} = (C \circ H_\varphi)_h.$$
  \end{itemize}
\end{lemma}
\begin{proof} (a):  The preliminary trace $\gamma_0$ operates on continuous functions on $R$ by restricting them to $\partial_h X$. Hence, for $f \in \ow{D}\cap \ell^\infty(X)$ we obtain
 $$\Tr (C\circ f) = \gamma_0 (C\circ f)  = C \circ (\gamma_0 f) =  C \circ (\Tr f).$$
 It remains to extend this equality to unbounded functions. Let $f \in \ow{D}$ and, for $n \geq 0$, let $f^{(n)}:= (f \wedge n)\vee (-n)$. Since $f^{(n)} \to f$
 with respect to $\|\cdot\|_o$, as $n \to \infty$, by Proposition~\ref{proposition:algebraic and order properties of domains}, the continuity of the trace map and the Lipschitz continuity of $C$ imply
  $$ C\circ \Tr f^{(n)} \to C \circ \Tr f, \text{ in } L^2(\partial_h X,\mu),\, \text{ as } n \to \infty.$$
  Moreover, we have $C \circ  f^{(n)} \to C \circ f$ pointwise, as $n \to \infty$, and
  $$\ow{Q}(C \circ f^{(n)}) \leq \ow{Q}(f).$$
  Lemma~\ref{lemma:q-wekly convergent sequences} yields $C \circ f^{(n)} \to C \circ f$ $\ow{Q}$-weakly, as $n \to \infty,$ and hence weakly in $(\ow{D},\|\cdot\|_o)$.
  Since  $(\ow{D},\|\cdot\|_o)$ is a Hilbert space, the Banach-Saks theorem implies the existence of a subsequence $(C \circ f^{(n_k)})$ such that
  $$\frac{1}{N}\sum_{k=1}^N C \circ f^{(n_k)} \to C \circ f, \text{ in } (\ow{D},\|\cdot\|_o), \text{ as } N \to \infty. $$
  Combining this observation with the continuity of the trace and what we have already proven yields
  \begin{align*}
   \Tr (C \circ f) &= \lim_{N\to \infty} \frac{1}{N}\sum_{k=1}^N \Tr (C \circ f^{(n_k)}) \\
  &=    \lim_{N\to \infty} \frac{1}{N}\sum_{k=1}^N C \circ \Tr f^{(n_k)}\\
  &= C \circ \Tr f.
  \end{align*}
  (b): The inclusion $\Dzero \subseteq \ker \Tr$ follows from the inclusion $C_c(X) \subseteq D_0 \cap \ell^\infty(X) \subseteq \ker \gamma_0$,
  see Corollary~\ref{corollary:d0 as kernel of gamma0},  and the continuity of $\Tr$.  Let $f \in \ker \Tr$ be given. Using the notation as in the proof of (a), and statement (a), we obtain
 $$0 = (\Tr f)^{(n)}  = \Tr f^{(n)} = \gamma_0 f^{(n)}, \text{ for all } n \geq 0.$$
  According to Corollary~\ref{corollary:d0 as kernel of gamma0}, this implies $f^{(n)} \in  \Dzero$ for all $n \geq 0$. Since $f^{(n)} \to f$ in $(\ow{D},\|\cdot\|_o)$, as
  $n\to \infty$, and since $\Dzero$ is a closed subspace, we infer $f \in \Dzero$.

 (c): Let $f \in \ow{D}$ and let $(f_n)$ a sequence in $\ow{D} \cap \ell^\infty(X)$ that converges to $f$ with respect to $\|\cdot\|_o$. The continuity of $\Tr$ and the continuity of the harmonic extension yields $H_{\Tr f_n} \to H_{\Tr f}$ pointwise. Moreover, Lemma~\ref{lemma:poinwise continuity royden decomposition} yields $(f_n)_h \to f_h$ pointwise. With these observations, the equality $H_{\Tr f} = f_h$ can be directly inferred from Proposition~\ref{Hgamma}, which shows $H_{\Tr f_n} = H_{\gamma_0 f_n} = (f_n)_h$.

 (d):  For $\varphi = \Tr f$ with $f \in \ow{D}$ assertions (b) and (c) imply
 $$\varphi = \Tr f = \Tr(f_0  + f_h) = \Tr f_h = \Tr H_{\Tr f} = \Tr H_\varphi.$$
 Moreover, for a normal contraction $C:\IR \to \IR$ assertions (a),(c) and the identity $\varphi = \Tr H_\varphi$ yield
 $$H_{C \circ \varphi} = H_{C \circ (\Tr H_\varphi) }  = H_{\Tr (C \circ H_\varphi)} = (C\circ H_\varphi)_h.$$
 This finishes the proof.
\end{proof}
\section{Traces of Dirichlet forms}\label{5}

In this section  we introduce the concept of traces of Dirichlet forms on graphs.

\begin{definition}[Trace Dirichlet form]\label{definition:trace dirichlet form}
 Let $(b,c)$ be a transient graph over a discrete measure space $(X,m)$ with $\partial_h X \neq \emptyset$ and let $\mu$ be a harmonic measure. The \emph{trace of a Dirichlet form}
 $Q$ on $\ell^2(X,m)$  is the quadratic form $\Tr Q$ on $L^2(\partial_h X,\mu)$
 with domain
 $$D(\Tr Q) := \Tr D(\Qe),$$
 on which it acts by
 $$\Tr Q(\varphi,\psi) := \Qe(H_\varphi,H_\psi).$$
 The trace of the form $\NeuDir$ is called the {\em Dirichlet to Neumann form} and is denoted by $$q^{DN}:=\Tr\NeuDir.$$
\end{definition}

Once one has a (good) trace map and a (good) harmonic extension of functions living one some auxiliary set (in our case  the harmonic boundary) this notion of trace Dirichlet forms is standard, see the following remarks.

\begin{remarks} \label{remark:our trace v.s. fukushima trace}
Under suitable assumptions traces of Dirichlet forms can be defined in a much more general context.
This is carried out in \cite[Chapter~5]{CF} where quasi-regular Dirichlet forms $ Q $ are considered. For suitable (i.e., quasi-closed) subsets $ F $ one can define restrictions for suitable (i.e., quasi-continuous) functions. Furthermore, for suitable (i.e., quasi-continuous) functions on this subset $ F $ one can define the harmonic extension via the associated stochastic process. In turn the trace of $ Q $ is then obtained by these harmonic extensions.
Our approach is related but it is nontrivial to prove that both definitions agree. Furthermore, our approach has the advantage that it is purely analytic and does not involve potential-theoretic quasi-notions. Moreover, if one knows the harmonic boundary and the harmonic measures our definition is more explicit.

\end{remarks}

\begin{remarks}
 On other graph boundaries the form $q^{DN}$ has appeared before at several places. In \cite{Sil}, $q^{DN}$ is considered as a regular Dirichlet form on the Martin boundary of the underlying graph and its  Beurling-Deny decomposition is derived. In particular, he proves that $q^{DN}$ is a jump form. In \cite{kasue2},  $q^{DN}$ is  considered as a form on the Kuramochi boundary. In both cases, $q^{DN}$ is defined as above with appropriately chosen trace map and harmonic extension.
\end{remarks}


As can be seen in Theorem~\ref{theorem:trace is a dirichlet form}
below, the trace of a Dirichlet form is in general only a Dirichlet
form in the wide sense. Recall that a \emph{Dirichlet form in the
wide sense} or Dirichlet form i. t. w. s. for short  is defined in
\cite{FOT} as a closed, Markovian form which is not necessarily
densely defined.

\subsection{The main result}
The following theorem is the main theorem of this paper. Its first part
says that every Dirichlet form $Q$ between $\QD$ and $\QN$
can be decomposed into an inner part (acting on functions on $X$)
and a boundary part (acting on functions on $\partial_h X$). Moreover,
its second part says that certain natural forms on the boundary yield Dirichlet
forms on $X$ via ``adding'' $\QD$. If the measure $m$ is finite this
yields that there is a class of Dirichlet forms in the wide sense on
the boundary that parametrizes the Dirichlet forms between $\QD$ and
$\QN$.

\medskip

In order to conveniently state the theorem we define for a graph
$(b,c)$ over $(X,m)$  the set of Dirichlet forms between $\QN$ and
$\QD$ by
$$\mathcal{D}(X,b,c,m):= \{Q \text{ Dirichlet form on } \ell^2(X,m) \text{ with } \QD \geq Q \geq
\QN\}$$ and the set of Dirichlet forms on the boundary by
$$\mathcal{D} (\partial X,b,c,m):=\{q \geq q^{DN} \text{ Dirichlet form i.~t.~w.~s. on } L^2(\partial_h X,\mu) \text{ with } q - q^{DN} \text{
Markovian}\}.$$

\begin{theorem}\label{theorem:main}
Let $(b,c)$ be a transient graph over the discrete measure space
$(X,m)$ with $\partial_h X \neq \emptyset$ and let $\mu$ be a
harmonic measure.
  \begin{itemize}
   \item[(a)] For  all Dirichlet forms $Q\in \mathcal{D}(X,b,c,m)$
   the form $q :=\Tr Q$ belongs to $\mathcal{D} (\partial X,b,c,m)$
   and with $q':= q -q^{DN}$ we have

      $$Q(f) = \QDe(f_0) + q(\Tr f) = \ow{Q}(f_0) + q (\Tr f)$$
   and
   $$Q(f) = \QN(f) + q'(\Tr f) = \ow{Q}(f) + q'(\Tr f).$$
   \item[(b)] For all Dirichlet forms in the wide sense $q \in \mathcal{D} (\partial
   X,b,c,m)$ the quadratic form
   $Q_q : \ell^2(X,m) \to [0,\infty]$ defined by
   $$Q_q (f):= \begin{cases}
   \ow{Q}(f_0) + q(\Tr f) &\text{if } f \in D(\QN) \text{ and } \Tr f \in D(q)\\
   \infty &\text{else}
   \end{cases} $$
   is a Dirichlet form with $\QD \geq Q_q \geq \QN$. Moreover, if $m$ is finite,
   the form $Q_q$  satisfies $q = \Tr Q_q$.
  \end{itemize}
\end{theorem}

The theorem has the following corollary.

\begin{corollary}\label{coro-parametrization}
 In the situation of the theorem the map
$$ \Tr:\mathcal{D}(X,b,c,m) \longrightarrow \mathcal{D}
 (\partial X,b,c,m)$$
is injective and, if $m$ is finite, also surjective.
\end{corollary}

The remainder of this section is dedicated to the proof of the preceding theorem and its corollary.
The proof of part (a) is essentially based on two theorems,
namely Theorem \ref{theorem:trace is a dirichlet form} and Theorem
\ref{theorem:q-qn is markovian}. These two theorems and their proofs
are the main content of the subsequent two sections. The proof of
part (b) is essentially based on the properties of the harmonic
extensions that are discussed in Section~\ref{4}. Before we start
with the proof, we include a few remarks.

\begin{remarks}[Relation to existing literature] \label{remark:ende}
 For pairs of regular Dirichlet forms $\mathcal{E}$ and $\tilde{\mathcal{E}},$ where
  $\tilde{\mathcal{E}}$ is a Silverstein extension of $\mathcal{E}$ this type of decomposition and its probabilistic interpretation are studied in \cite[Chapter 7]{CF}.
  In contrast to this classical approach we consider a whole family of forms associated to a graph instead of pairs of Dirichlet forms.
  Therefore, the boundary on which we decompose forms is a genuine geometric object associated with the given graph. Moreover, we give a precise characterization of
  the image of the trace map on forms when $m$ is finite. Similar results have only been obtained recently in \cite{posi}
  for extensions of elliptic operators on bounded domains in $\IR^d$. However, our methods are different from the ones used in \cite{posi}
  and \cite{CF}. We directly use form methods while \cite{posi} uses Krein's resolvent formula and \cite{CF} uses a mixture of probabilistic and analytic arguments
  that are based on the Beurling-Deny representation of regular Dirichlet forms.
\end{remarks}

\begin{remarks}[Necessity of finite measure condition] The assumption of finite measure in the last part of the theorem above should not come as a surprise. Assume in sharp contrast that
 there is a constant $C >0$ such that $m$ satisfies $m(x) \geq C$ for all $x \in X$. Then all functions in $\ell^2(X,m)$ vanish at infinity. In particular, $\Tr f = 0$
 for any function  $f \in \ell^2(X,m) \cap \ow{D}$. In this situation nontrivial Dirichlet forms on $L^2(\partial_h X,\mu)$ cannot be realized as traces of forms on $\ell^2(X,m)$.
 Hence, the map $\Tr$ is in general not surjective.
\end{remarks}

\begin{remarks}[Topological character of the theorem]
The given proofs and the definition of the trace form show that the previous theorem is rather a theorem about extended Dirichlet forms considered as quadratic
forms on $C(X)$ (without a fixed reference measure) than about Dirichlet forms. This explains why we  obtain a full characterization of the image of the trace map when $m$
is finite. In this case, bounded functions in the domain of the extended Dirichlet form belong to the domain of the Dirichlet form and so the theory of extended Dirichlet
forms and of the Dirichlet forms from which they derive is basically the same.
\end{remarks}

\subsection{Traces of Dirichlet forms are Dirichlet forms in the wide sense}
In this section we prove that traces of Dirichlet forms are Dirichlet forms in the wide sense. In particular, we give a purely analytic proof for (part of) \cite[Theorem~5.2.2]{CF} in the framework of infinite graphs (see Remark~\ref{remark:our trace v.s. fukushima trace} as well).

\begin{theorem} \label{theorem:trace is a dirichlet form}
 Let $(b,c)$ be a transient graph over a discrete measure space $(X,m)$ with $\partial_h X \neq \emptyset$ and let $\mu$ be a harmonic measure.
 For every Dirichlet form $Q$ on $\ell^2(X,m)$ that satisfies $\QD \geq Q \geq \QN$ the trace form $\Tr Q$ is a Dirichlet form in the wide sense on $L^2(\partial_h X,\mu)$.
 In particular, $q^{DN}$ is a Dirichlet form in the wide sense on $L^2(\partial_h X,\mu)$. Moreover, if $m$ is finite, then $q^{DN}$ is densely defined.
\end{theorem}
For $\varphi = \Tr f$ with $f \in D(\Qe)\subseteq \ow{D}$ we have $H_\varphi = H_{\Tr f} = f_h$ by Lemma~\ref{proposition:properties of trace}.
Therefore, the following lemma implies that our definition of $\Tr Q$ is appropriate, as it shows that the Royden decomposition is also an orthogonal decomposition with respect to the extended space of a form between $\QD$ and $\QN$.

\begin{lemma} \label{lemma:royden decomposition in dqe}
 Let $(b,c)$ be a transient graph over a discrete measure space $(X,m)$ with $\partial_h X \neq \emptyset$. Let $Q$ be a Dirichlet form on $\ell^2(X,m)$ that satisfies $\QD \geq Q \geq \QN$. For all $f \in D(Q_{\rm e})$ we have $f_0,f_h \in D(\Qe)$ and the identity
 $$\Qe(f) = \QDe(f_0) + \Qe(f_h) = \ow{Q}(f_0) + \Qe(f_h)$$
 holds. In particular, for $f \in D(\Qe)$ and $g \in D(\QDe) = \Dzero$ we have
 $$\Qe(f,g) = \QNe(f,g) = \ow{Q}(f,g).$$
\end{lemma}
\begin{proof}
 As shown in Lemma~\ref{lemma:extended dirichlet spaces of qd and qn},  $\QD \geq Q \geq \QN$ implies
 $$\Dzero = D(\QD) \subseteq D(\Qe) \subseteq D(\QNe) \subseteq \ow{D}$$
 and so the Royden decomposition theorem can be applied to $f \in D(\Qe)$. We obtain $f_0 \in \Dzero = D(\QDe) \subseteq D(\Qe)$ and so $f_h  = f - f_0 \in D(\Qe)$. Since
 $$\Qe(f) = \Qe(f_0) + 2 \Qe(f_0,f_h) + \Qe(f_h) = \QDe(f_0) + 2 \Qe(f_0,f_h) + \Qe(f_h) ,$$
 we are left to prove $\Qe(f_0,f_h) = 0$. For all $\alpha \in \IR$ the inequality
 $$\Qe(\alpha f_0 + f_h) \geq \QNe(\alpha f_0 + f_h)$$
 holds. Expanding this inequality and using that $\Qe$ and $\QNe$ agree on $\Dzero = D(\QDe)$, we obtain
 $$2\alpha \Qe(f_0,f_h) + \Qe(f_h) \geq 2\alpha \QNe(f_0,f_h) + \QNe(f_h).$$
 Since $\QNe$ is a restriction of $\ow{Q}$, Lemma~\ref{orthharm} implies $\QNe(f_0,f_h) = 0.$ But then the inequality
 $$2\alpha \Qe(f_0,f_h) + \Qe(f_h) \geq \QNe(f_h)$$
 can only hold true for all $\alpha \in \IR$ if $\Qe(f_0,f_h) = 0$.

 For the ``in particular'' part, we let $f \in D(\Qe)$ and $g \in D(\QDe)$ and obtain
 $$\Qe(f,g) = \Qe(f_0,g) + \Qe(f_h,g).$$
 By what we have already proven $\Qe(f_h,g) = \Qe(f_h,g_0) = 0$ follows (apply the derived formula to the function $g_0 + f_h$). Thus, an application of Lemma~\ref{lemma:extended dirichlet spaces of qd and qn} and the Royden decomposition theorem show
 $$\Qe(f,g) = \Qe(f_0,g) = \ow{Q}(f_0,g) = \ow{Q}(f,g) = \QNe(f,g).$$
 This finishes the proof.
\end{proof}

\begin{proof}[Proof of Theorem~\ref{theorem:trace is a dirichlet form}]
 We first show that $\Tr Q$ is closed by proving that $D(\Tr Q)$ equipped with the inner product
 $$\as{\cdot,\cdot}_{\Tr Q} := \Tr Q + \as{\cdot,\cdot}_{L^2(\partial_h X,\mu)}$$
 is a Hilbert space. Let $(\varphi_n)$ be Cauchy with respect to $\as{\cdot,\cdot}_{\Tr Q}$ and let $\varphi$ be its limit in $L^2(\partial_h X,\mu)$. According to Lemma~\ref{lemma:pointwise continuity harmonic extension}, we have $H_{\varphi_n} \to H_\varphi$ pointwise and by the definition of $\Tr Q$ the sequence $(H_{\varphi_n})$ is $\Qe$-Cauchy. The lower semicontinuity of extended Dirichlet spaces with respect to pointwise convergence, see Proposition~\ref{proposition:pointwise lower semicontinuity extended space}, yields
 $$\Qe(H_\varphi - H_{\varphi_m}) \leq \liminf_{n\to \infty} \Qe(H_{\varphi_n} - H_{\varphi_m}).$$
 Therefore, $H_\varphi\in D(\Qe)$ and $H_{\varphi_n} \to H_\varphi$ with respect to $\Qe$. It remains to show that $\varphi \in D(\Tr Q) = \Tr D(\Qe)$,
 which will follow from  $\varphi = \Tr H_\varphi$. Note that this identity can not be directly inferred from Lemma~\ref{proposition:properties of trace}~(d),
 since its assertion is only valid for functions in $\Tr \ow{D}$ but we only assumed $\varphi\in L^2(\partial_h X,\mu)$. According to Lemma~\ref{lemma:extended dirichlet spaces of qd and qn}, convergence with respect to $\Qe$ implies convergence with respect to $\ow{Q}$ and so $H_{\varphi_n} \to H_\varphi$ with respect to $\|\cdot\|_o$. The continuity of the trace and Lemma~\ref{proposition:properties of trace}~(d) applied to $\varphi_n \in \Tr D(\Qe) \subseteq \Tr \ow{D}$ yields
 $$\Tr H_\varphi = \lim_{n \to \infty} \Tr H_{\varphi_n} = \lim_{n \to \infty} \varphi_n = \varphi\mbox{ in }L^2(\partial_h X,\mu).$$
 As mentioned above, this shows closedness of $\Tr Q$.

 For proving the Markov property of $\Tr Q$, we let $\varphi \in D(\Tr Q)$ and let $C:\IR \to \IR$ a normal contraction. Since $\varphi \in \Tr D(\Qe) \subseteq \Tr \ow{D}$,
 Lemma~\ref{proposition:properties of trace} (d) yields $(C\circ H_\varphi)_h = H_{C \circ \varphi}$. With this identity at hand,
 Lemma~\ref{lemma:royden decomposition in dqe} shows $H_{C \circ \varphi} \in D(\Qe)$ and
 $$\Tr Q(\varphi) = \Qe(H_\varphi) \geq \Qe(C \circ H_\varphi) =  \Qe(H_{C \circ \varphi}) + \Qe((C\circ H_\varphi)_0) \geq  \Tr Q(C \circ \varphi). $$
 Here, we used that the extended Dirichlet forms also possess the Markov property.

 If $m$ is finite, we have $\ow{D}\cap \ell^\infty(X) \subseteq \ow{D} \cap \ell^2(X,m) = D(\QN)$ and therefore $\Tr (\ow{D}\cap \ell^\infty(X)) \subseteq D(q^{DN})$.  The algebra $\Tr (\ow{D}\cap \ell^\infty(X)) = \gamma_0 (\ow{D}\cap \ell^\infty(X))$ is uniformly dense in $C(\partial_h X)$ or in $\{C(\partial_h X) \, : \, f(p_\infty) = 0\}$, depending on whether or not $1 \in \mathcal{A}$, cf. Lemma~\ref{Adarst}. Since $\mu$ is a finite regular Borel measure on $\partial_h X$ (with $\mu(\{p_\infty\}) = 0$ if $1 \not \in \mathcal{A}$), this implies the density of $D(q^{DN}) = D(\Tr \QN)$ in $L^2(\partial_h X,\mu)$.
\end{proof}
\subsection{Differences of Dirichlet Forms are Markovian} \label{section:differences of dirichlet forms}

It is the main goal of this section to prove that  for a form $Q$ between $Q^{(D)}$ and $Q^{(N)}$ the difference $Q_e-Q_e^{(N)}$, considererd as a quadratic form on $D(Q)$, is Markovian.
The strategy for the presented proof and the involved formulas are taken from \cite[Chapter~3]{Sch}, which treats more general energy forms on
possibly non-discrete spaces.

\begin{theorem} \label{theorem:q-qn is markovian}
 Let $(b,c)$ be a graph over the discrete measure space $(X,m)$ and let $Q$ be a Dirichlet form on $\ell^2(X,m)$ with $\QD \geq Q \geq \QN$. Then the quadratic form $\Qe - \QNe$ with domain $D(\Qe)$ is Markovian.
\end{theorem}

Before proving the theorem we note the following corollary, which states that the difference of the trace forms on the harmonic boundary is Markovian as well.

\begin{corollary}\label{corollary: q-qdn is markovian}
 Let $(b,c)$ be a transient graph over the discrete measure space $(X,m)$ with $\partial_h X \neq \emptyset$ and let $\mu$ be a harmonic measure. For all
 Dirichlet forms $Q$ on $\ell^2(X,m)$ with $\QD \geq Q \geq \QN$ the quadratic form $\Tr Q - q^{DN}$ with domain $D(\Tr Q)$ is a Markovian form on $L^2(\partial_h X,\mu)$.
\end{corollary}
\begin{proof}
 Let $\varphi \in D(\Tr Q)$ and let $C:\IR \to \IR$ a normal contraction. Since $\Tr Q$ is a Dirichlet form in the wide sense, we have $C \circ \varphi \in D(\Tr Q)$.
 Using the Markov property of $\Qe - \QNe\geq 0$ and the same reasoning as in the proof of Theorem~\ref{theorem:trace is a dirichlet form}, we obtain
 \begin{align*}
  \Tr Q(\varphi)  - q^{DN} (\varphi) &= \Qe(H_\varphi) - \QNe(H_\varphi)\\
  &\geq \Qe(C\circ H_\varphi) - \QNe (C\circ H_\varphi) \\
  &= \Qe(H_{C \circ \varphi}) - \QNe (H_{C \circ \varphi}) + \Qe((C\circ H_\varphi)_0) - \QNe ((C\circ H_\varphi)_0)\\
  &\geq \Qe(H_{C \circ \varphi}) - \QNe (H_{C \circ \varphi})\\
  &= \Tr Q(C\circ \varphi) - q^{DN}(C\circ \varphi).
 \end{align*}
This finishes the proof.
\end{proof}

Let $(X,m)$ be a discrete measure space and let $Q$ be a Dirichlet form on $\ell^2(X,m)$. For $\varphi \in D(Q)$ with $0 \leq \varphi \leq 1$ we define the quadratic form
$$Q_\varphi:D(Q) \cap \ell^\infty(X) \to \IR,\, Q_\varphi(f):= Q(\varphi f) - Q(\varphi f^2,\varphi).$$
\begin{remarks}
 The definition of $Q_\varphi$ makes sense for functions $f \in \ell^\infty(X)$ that satisfy $\varphi f, \varphi f^2 \in D(Q)$. When considered on this larger domain, $Q_\varphi$ is a closable Markovian quadratic form on $C(X)$ equipped with the topology of pointwise convergence. Properties of its closure have been studied in \cite[Chapter~3]{Sch} in the context of maximal Silverstein extensions.
\end{remarks}

\begin{remarks}
If $Q$ is a restriction of $\ow{Q}$ for a graph $(b,c)$ such that $\varphi\equiv 1\in D(Q)$, then a simple calculation shows that $Q_\varphi(f)=\frac12 \sum_{x,y\in X}b(x,y)(f(x)-f(y))^2$
and $Q-Q_\varphi(f)=\sum_{x\in X} c(x)f(x)^2$ hold. Hence, this construction allows us to substract the killing part in an abstract way.
\end{remarks}

\begin{lemma} \label{lemma:resolvent approximation qphi}
 Let $(X,m)$ be a discrete measure space, let $Q$ be a Dirichlet form on $\ell^2(X,m)$. There exists a family of weighted graphs $(b_\alpha,c_\alpha)$, $\alpha>0$,
 over $X$ such that for all $f \in D(Q)\cap \ell^\infty(X)$ and all $\varphi \in D(Q)$ with $0\leq \varphi \leq 1$ we have
 $$Q_\varphi(f) = \lim_{\alpha\to\infty} \frac12\sum_{x,y\in X}b_\alpha(x,y) \varphi(x)\varphi(y)(f(x)-f(y))^2$$
 and
 $$Q(f) - Q_\varphi(f) = \lim_{\alpha\to\infty} \left[ \frac12\sum_{x,y\in X}b_\alpha(x,y) (1- \varphi(x)\varphi(y))(f(x)-f(y))^2 + \sum_{x \in X} c_\alpha(x) f(x)^2 \right].$$
\end{lemma}
\begin{proof}
 Let $G_\alpha,\alpha > 0$, denote the resolvent of the form $Q$ and let
 $$Q^{(\alpha)}:\ell^2(X,m) \to \IR,\, Q^{(\alpha)}(f) =  \alpha \as{(I-\alpha G_\alpha)f,f}$$
 be the associated approximating form. It is well-known, see e.g. \cite[Section~1.4]{FOT}, that $Q^{(\alpha)}$ is a continuous Dirichlet form on $\ell^2(X,m)$ that satisfies
 $$Q(f,g) = \lim_{\alpha \to \infty} Q^{(\alpha)}(f,g), $$
 for all $f,g \in D(Q)$. In particular, $Q^{(\alpha)}$ is regular and so \cite[Theorem~7]{stoch} implies the existence of a weighted graph $(b_\alpha,c_\alpha)$ over $X$ such that for all $f \in \ell^2(X,m)$ the value $Q^{(\alpha)}(f)$ is given by
 $$Q^{(\alpha)}(f) = \frac12\sum_{x,y\in X} b_\alpha(x,y)(f(x) - f(y))^2 + \sum_{x \in X} c_\alpha(x) f(x)^2.$$
 With this at hand, a  straightforward computation that uses the definition of $Q_\varphi$ proves the claim.
 \end{proof}
The next lemma is a special case of \cite[Lemma~3.39]{Sch} and shows some properties of $Q_\varphi$.  It is a direct consequence of the previous lemma.
\begin{lemma}\label{lemma:properties qphi}
Let $(X,m)$ be a discrete measure space, let $Q$ be a Dirichlet form on $\ell^2(X,m)$ and let $\varphi,\psi \in D(Q)$ with $0\leq \varphi,\psi \leq 1$.
\begin{itemize}
 \item[(a)] For all $f \in D(Q)\cap \ell^\infty(X)$ we have $Q_\varphi(f) \leq Q(f)$.
 \item[(b)] If $\psi \leq \varphi$, then $Q_\psi(f) \leq Q_\varphi(f)$ for all $f \in D(Q)\cap \ell^\infty(X)$.
 \item[(c)] $Q_\varphi$ is Markovian, i.e. for all $f \in D(Q)\cap \ell^\infty(X)$ and all normal contractions $C:\IR\to\IR$ we have
 $$Q_\varphi(C \circ f) \leq Q_\varphi(f).$$
\end{itemize}
\end{lemma}
For a Dirichlet form $Q$ on $\ell^2(X,m)$ the \emph{main part}
$$Q^M(f):= \sup\{Q_\varphi(f) \mid \varphi \in D(Q) \text{ with } 0 \leq \varphi \leq 1\},$$
and the \emph{killing part}
$$Q^k(f) := Q(f) - Q^M(f), $$
 induce maps $Q^M:D(Q)\cap\ell^\infty(X) \to [0,\infty)$ and $Q^k:D(Q)\cap\ell^\infty(X) \to [0,\infty)$. The previous lemma implies the following important observation for $Q^M$ and
 $Q^k$.
 \begin{lemma}
  Let $(X,m)$ be a discrete measure space and let $Q$ be a Dirichlet form on $\ell^2(X,m)$. Then, the main part $Q^M$ and the killing part $Q^k$ are Markovian quadratic forms.
 \end{lemma}
\begin{proof}

 The monotonicity of the forms $Q_\varphi$ in the parameter $\varphi$, Lemma~\ref{lemma:properties qphi}~(b), implies that $Q^M$  is a quadratic form.
 Thus, $Q^k$ is a quadratic form as well. Moreover, $Q^M$ is Markovian by Lemma~\ref{lemma:properties qphi}~(c) and the Markov property of $Q^k$ follows from Lemma~\ref{lemma:resolvent approximation qphi}.
\end{proof}

Let $(b,c)$ be a graph  over the discrete measure space $(X,m)$. Recall that we denote by $Q^{(N)}_{(b,0)}$  the Neumann form of the graph $(b,0)$ and by $Q^{(N)}_{(0,c)}$ the Neumann form of the (totally disconnected) graph $(0,c)$. The next lemma shows that $Q^{(N)}_{(b,0)}$ (restricted to $D(\QN)\cap \ell^\infty(X)$) is the main part of $\QN$ and that  $Q^{(N)}_{(0,c)}$ (restricted to $D(\QN)\cap \ell^\infty(X)$) is the killing part of $\QN$.
\begin{lemma} \label{lemma:formula for QN  with exhaustion}
 Let $(b,c)$ be a graph over the discrete measure space $(X,m)$ and let $(\varphi_n)$ be an increasing sequence in $D(\QN)$ with $0\leq \varphi_n \leq 1$ that  converges pointwise to the constant function $1$. For all $f \in D(\QN) \cap \ell^\infty(X)$ the identity
 $$Q^{(N)}_{(b,0)}(f) = \lim_{n \to \infty} Q^{(N)}_{\varphi_{n}}(f) = \lim_{n \to \infty} \left[ Q^{(N)} (\varphi_n f) - \QN(\varphi_n f^2,\varphi_n)\right]  $$
 holds.
\end{lemma}
\begin{proof}
For $f \in D(\QN) \cap \ell^\infty(X)$ and $\varphi \in D(\QN)$ with $0 \leq \varphi \leq 1$  a direct computation shows
$$Q^{(N)} (\varphi f) - \QN(\varphi f^2,\varphi) = \frac12\sum_{x,y \in X} b(x,y) \varphi(x) \varphi(y) (f(x)- f(y))^2.$$
Hence, we obtain the statement after applying the monotone convergence theorem.
\end{proof}
\begin{lemma}\label{lemma: choosing increasing seq}
 Let $(b,c)$ be a graph over the discrete measure space $(X,m)$ and let $Q$ be a Dirichlet form on $\ell^2(X,m)$ with $\QD \geq Q \geq \QN$. Let $f,g\in D(Q)\cap\ell^\infty(X)$.
 Then, there exists a sequence $(\varphi_n)$ in $D(Q)\cap\ell^\infty(X)$ such that $0\leq \varphi_n\leq 1$ for every $n\in\IN$, $(\varphi_n)$ is pointwise monotonically increasing
 and pointwise convergent to 1, and $$Q^M(f)=\lim_{n\to\infty} Q_{\varphi_n}(f),\quad Q^M(g)=\lim_{n\to\infty} Q_{\varphi_n}(g)$$ hold.
\end{lemma}
\begin{proof}
 Let $f,g\in D(Q)\cap\ell^\infty(X)$. Then, by definition of $Q^M$, there are sequences $(\varphi_n^{(1)})$ and $(\varphi_n^{(2)})$ in $D(Q)\cap\ell^\infty(X)$,
 $0\leq \varphi_n^{(i)}\leq 1$, $i=1,2$, such that
 $Q^M(f)=\lim_{n\to\infty} Q_{\varphi_n^{(1)}}(f)$ and $Q^M(g)=\lim_{n\to\infty} Q_{\varphi_n^{(2)}}(g)$ hold. By Lemma \ref{lemma:properties qphi} (b) these sequences can be chosen to be pointwise
 monotonically increasing and pointwise convergent to $1$.
 We define $$\varphi_n:=\varphi_n^{(1)}\vee \varphi_n^{(2)}$$ and the result follows by Lemma \ref{lemma:properties qphi} (b) and since
 $D(Q)\cap\ell^\infty(X)$ is a lattice by Proposition~\ref{proposition:algebraic and order properties of domains}.
\end{proof}

\begin{lemma} \label{lemma:computing qk and qm for compactly supported functions}
 Let $(b,c)$ be a graph over the discrete measure space $(X,m)$ and let $Q$ be a Dirichlet form on $\ell^2(X,m)$ with $\QD \geq Q \geq \QN$.
 For all $f \in D(Q) \cap \ell^\infty(X)$ and $g \in C_c(X)$ we have %
 $$Q^M(f,g) =  Q^{(N)}_{(b,0)}(f,g) \text{ and } Q^k(f,g) =  Q^{(N)}_{(0,c)}(f,g).$$
\end{lemma}
\begin{proof}
 Lemma~\ref{lemma:royden decomposition in dqe} shows that for $f\in D(Q) \cap \ell^\infty(X)$ and $g \in C_c(X)$ we have
 $$Q(f,g) = \QN(f,g).$$
 Therefore, it suffices to prove the identity $Q^M(f,g) =  Q^{(N)}_{(b,0)}(f,g)$. By Lemma \ref{lemma: choosing increasing seq} (and the polarization identity) we choose an increasing, pointwise to $1$ convergent
 sequence $(\varphi_n)$ in $D(Q)$ with $0 \leq \varphi_n \leq 1$ and
 $$Q^M(f,g) = \lim_{n\to \infty}Q_{\varphi_n}(f,g).$$
 With this at hand, Lemma~\ref{lemma:royden decomposition in dqe} and Lemma~\ref{lemma:formula for QN  with exhaustion} show
 \begin{align*}
  Q^M(f,g) &= \lim_{n\to \infty}Q_{\varphi_n}(f,g)\\
  &=\lim_{n\to \infty} \left[  Q(\varphi_n f, \varphi_n g) -  Q(\varphi_n fg,\varphi_n) \right]\\
  &= \lim_{n\to \infty} \left[  \QN(\varphi_n f, \varphi_n g) -  \QN(\varphi_n fg,\varphi_n) \right]\\
  &= Q^{(N)}_{(b,0)} (f,g).
 \end{align*}
Note that for the third equality we used that $\varphi_n g$ and $\varphi_n fg$ have finite support and, therefore, belong to $D(\QD)$. This finishes the proof.
\end{proof}

\begin{lemma} \label{lemma:difference of main and killing parts are Markovian}
 Let $(b,c)$ be a graph over the discrete measure space $(X,m)$ and let $Q$ be a Dirichlet form on $\ell^2(X,m)$ with $\QD \geq Q \geq \QN$. The quadratic forms $ Q^M - Q^{(N)}_{(b,0)}$  and $ Q^k -  Q^{(N)}_{(0,c)}$ with domain $D(Q) \cap \ell^\infty(X)$ are Markovian.
\end{lemma}
\begin{proof}   Let $f \in D(Q) \cap \ell^\infty(X)$ and let $C :\IR \to \IR$ a normal contraction. Since $Q$ is a Dirichlet form, we have $C \circ f \in D(Q)\cap \ell^\infty(X)$ and so we only need to establish the inequalities
$$Q^M(C\circ f)  - Q^{(N)}_{(b,0)}(C\circ f) \leq Q^M(f)  - Q^{(N)}_{(b,0)}(f)$$
and
$$Q^k(C\circ f) -  Q^{(N)}_{(0,c)}(C\circ f) \leq Q^k(f) -  Q^{(N)}_{(0,c)}(f).$$

Lemma~\ref{lemma: choosing increasing seq} implies that there exists an increasing sequence $(\psi_n)$ in $D(Q)$ with $0 \leq \psi_n \leq 1$ such that
$$Q^M(f) = \lim_{n\to \infty} Q_{\psi_n}(f) \text{ and } Q^M(C\circ f) = \lim_{n\to \infty} Q_{\psi_n}(C\circ f).$$
Moreover, this sequence can be chosen to converge  pointwise to the constant function $1$. We choose another increasing sequence $(\varphi_n)$ in $C_c(X)$ that converges pointwise  to the constant function $1$ and satisfies $\varphi_n \leq \psi_n$ for all $n$. According to Lemma~\ref{lemma:formula for QN  with exhaustion}, we obtain
\begin{align*}
 Q^M(f) -  Q^{(N)}_{(b,0)}(f) &= \lim_{n\to \infty} \left[Q_{\psi_n}(f) -  Q^{(N)} (\varphi_n f) + \QN(\varphi_n f^2,\varphi_n)  \right]\\
 &=  \lim_{n\to \infty} \left[Q_{\psi_n}(f) - Q_{\varphi_n}(f) \right],
\end{align*}
where we used that $Q$ and $\QN$ agree on $C_c(X)$ for the last equality. By the choice of $(\psi_n)$, the same formula is valid with $f$ being replaced by
$C \circ f$. Lemma~\ref{lemma:resolvent approximation qphi} shows that there exists a family of graphs $(b_\alpha,c_\alpha)$ over $X$ such that
$$Q_{\psi_n}(g) - Q_{\varphi_n}(g) = \lim_{\alpha \to \infty} \sum_{x,y \in X} b_\alpha(x,y) (\psi_n(x) \psi_n(y) - \varphi_n(x) \varphi_n(y))(g(x) - g(y))^2 $$
for every $g\in D(Q)\cap\ell^\infty(X)$.
Since $\psi_n(x) \psi_n(y) \geq  \varphi_n(x) \varphi_n(y)$, this identity and the previous computations imply the Markov property of  $Q^M - Q^{(N)}_{(b,0)}$.

For proving the result on $Q^k -  Q^{(N)}_{(0,c)}$, we let $(K_n)$ be an increasing sequence of finite subsets of $X$ with $\cup K_n = X$ and denote by $1_{K_n}$
the corresponding sequence of indicator functions. The monotone convergence theorem implies
$$Q^k(f) - Q^{(N)}_{(0,c)}(f) = \lim_{n\to \infty} \left[Q^k(f) - Q^{(N)}_{(0,c)}(1_{K_n} f) \right]$$ for $f\in D(Q)\cap\ell^\infty(X)$, since $1_{K_n}f\in C_c(X)$.
From Lemma~\ref{lemma:computing qk and qm for compactly supported functions} we infer
$$Q^k(1_{K_n}f, 1_{X \setminus K_n} f) = Q^{(N)}_{(0,c)}(1_{K_n}f, 1_{X \setminus K_n} f) = 0 \text{ and } Q^k(1_{K_n} f) = Q^{(N)}_{(0,c)}(1_{K_n} f).$$
Therefore, the previous computation simplifies to
$$Q^k(f) - Q^{(N)}_{(0,c)}(f) = \lim_{n\to \infty} Q^k( 1_{X \setminus K_n} f) $$
and the Markov property of $Q^k - Q^{(N)}_{(0,c)}$ follows from the Markov property of $Q^k$.
\end{proof}

\begin{proof}[Proof of Theorem~\ref{theorem:q-qn is markovian}]
  According to Lemma~\ref{lemma:difference of main and killing parts are Markovian}, and by $$Q-Q^{(N)}=(Q^M-Q_{(b,0)}^{(N)})+(Q^k-Q_{(0,c)}^{(N)}),$$
  the quadratic form $Q - \QN$ is Markovian on $D(Q) \cap \ell^\infty(X)$.
  It remains to prove that this property is stable when passing to the extended spaces. To this end, let $C:\IR \to \IR$ a normal contraction,
  let $f \in D(\Qe)$ and let $(f_n)$ be a $Q$-Cauchy sequence that converges pointwise to $f$. Since $D(Q) \cap \ell^\infty(X)$ is dense in $D(Q)$ with respect to the form norm,
  we can assume $f_n \in D(Q) \cap \ell^\infty(X).$ The inequality $\Qe \geq \QNe$ implies that $(f_n)$ converges to $f$ with respect to $\|\cdot\|_o$ and that the sequence
  $(C \circ f_n)$ is bounded in the Hilbert space $(\ow{D},\|\cdot\|_o)$. Hence, by the theorem of Banach-Saks there exists a subsequence $(f_{n_k})$ such that
  $$g_N := \frac{1}{N}\sum_{k = 1}^N ( C \circ f_{n_k})$$
  converges towards some $g \in \ow{D}$ with respect to $\|\cdot\|_o$. Since this convergence implies pointwise convergence and since $C \circ f_n \to C \circ f$ pointwise, we obtain $g_N \to C \circ f$ with respect to $\|\cdot\|_o$. The pointwise and $\|\cdot\|_o$-convergence of $(g_N)$ and the lower semicontinuity of extended forms with respect to pointwise convergence yields
  $$\Qe(C\circ f) - \QNe(C \circ f) \leq \liminf_{N\to \infty} \Qe(g_N) - \lim_{N\to \infty} \QNe(g_N) = \liminf_{N\to \infty} \left(\Qe(g_N) - \QNe(g_N)\right).   $$
  Using that $\Qe - \QNe$ is convex by $\frac{d}{dt}(\Qe-\QNe)(u+tv)=2(\Qe-\QNe)(v)\geq 0$ for every $u,v\in D(Q_e)$, we obtain
  $$\Qe(C\circ f) - \QNe(C \circ f) \leq  \liminf_{N\to \infty} \frac{1}{N} \sum_{k = 1}^N (\Qe(C\circ f_{n_k}) - \QNe(C\circ f_{n_k})).$$
  The Markov property of $\Qe - \QNe$ on $D(Q) \cap \ell^\infty(X)$ yields
  $$\Qe(C\circ f_{n_k}) - \QNe(C\circ f_{n_k}) \leq \Qe(f_{n_k}) - \QNe(f_{n_k}).$$
  Therefore, the inequality
  $$\Qe(C\circ f ) - \QNe(C\circ f) \leq \Qe(f) - \QNe(f),$$
  follows from the previous computations and the choice of the sequence $(f_n)$. This finishes the proof.
\end{proof}

\subsection{Proof of the main result}\label{section: proof main theorem}
\begin{proof}[Proof of Theorem~\ref{theorem:main}]
 (a): Let  $f \in D(Q)$.  Lemma~\ref{lemma:royden decomposition in dqe} implies
 $$Q(f) = \QDe(f_0) + \Qe(f_h) = \ow{Q}(f_0) + \Qe(f_h)$$
 and $$\QN(f) = \QDe(f_0) + \QNe(f_h) = \ow{Q}(f_0) + \QNe(f_h).$$
 Therefore, we infer
 $$Q(f) = \QN(f)  + Q(f) - \QN(f) = \QN(f) + \Qe(f_h) - \QNe(f_h)=\ow{Q}(f) + \Qe(f_h) - \QNe(f_h).$$
 Moreover, Lemma~\ref{proposition:properties of trace} (c) yields  $H_{\Tr f} = f_h$ and we obtain by definition of $\Tr Q$ and $q^{DN}$
 $$\Qe(f_h) = \Qe(H_{\Tr f}) = \Tr Q (\Tr f) \quad\text{and}\quad\QNe(f_h) = \QNe(H_{\Tr f}) = q^{DN}(\Tr f).$$
 These computations show that letting $q = \Tr Q$ and $q' = \Tr Q - q^{DN}$ yields a decomposition as in (a).
 Theorem~\ref{theorem:trace is a dirichlet form} yields that $q$ is a Dirichlet form in the wide sense and Corollary~\ref{corollary: q-qdn is markovian} yields that $q'$ is Markovian.

 (b): Let $q \geq q^{DN}$ be a Dirichlet form on $L^2(\partial_h X,\mu)$ such that $q - q^{DN}$ is Markovian. For $f\in D(\QN)$ we have
 $\ow{Q}(f_h) = \QNe(f_h)=\QNe(H_{\Tr f}) = q^{DN}(\Tr f)$ by Lemma~\ref{proposition:properties of trace} (c).  For $ f \in D(Q_q)$ this implies
 $$Q_q(f) = \ow{Q}(f_0) + \ow{Q}(f_h) + q(\Tr f) - q^{DN}(\Tr f) = \QN(f) + q(\Tr f) - q^{DN}(\Tr f) \geq \QN(f).$$
 This proves $Q_q \geq \QN$. For $g \in D(\QD)$ we have $g = g_0$ and $\Tr g = 0$ by Lemma~\ref{proposition:properties of trace}. We obtain
 $$Q_q(g) = \ow{Q}(g_0) = \ow{Q}(g) = \QD(g).$$
 Therefore, $\QD \geq Q_q$ is also satisfied.

 Next, we prove the Markov property of $Q_q$. Let $C :\IR \to \IR$ be a normal contraction and let $f \in D(Q_q)$.
 The Markov property of $q - q^{DN}$ and $\QN$ together with the identity $Q_q(g)=\QN(g) + q(\Tr g) - q^{DN}(\Tr g)$, $g\in D(Q_q)$, seen above show
 \begin{align*}
  Q_q(f) &= \QN(f) + q (\Tr f) - q^{DN}(\Tr f) \\
  &\geq \QN(C \circ f) + q(C \circ (\Tr f)) - q^{DN}(C \circ (\Tr f))\\
  &=  \QN(C \circ f) + q( \Tr (C\circ  f)) - q^{DN}(\Tr ( C\circ f))\\
  &= Q_q (C \circ f),
 \end{align*}
 where we used that $C$ and $\Tr$ commute on $\Tr D(Q_q) \subseteq \Tr \ow{D}$ by Lemma~\ref{proposition:properties of trace} (a).

 Next we show the closedness of $Q_q$. Let $(f_n)$ be a Cauchy sequence in $(D(Q_q),\|\cdot\|_{Q_q})$. We show that $(f_n)$ has a $\|\cdot\|_{Q_q}$-limit.
 The inequality $Q_q \geq \QN$ implies that the space $(D(Q_q),\|\cdot\|_{Q_q})$ continuously embeds in the Hilbert space $(\ow{D},\|\cdot\|_o)$.
 Let $f$ be the limit  of $(f_n)$ in $(\ow{D},\|\cdot\|_o)$. We have to show $\ow{Q}(f_0 - (f_n)_0)\to 0$, and $f\in D(\QN)$, and $q(\Tr f_n-\Tr f)\to 0$.

 The Royden decomposition, Theorem~\ref{theorem:royden decomposition}, shows
 $$\ow{Q}(f_0 - (f_n)_0) = \ow{Q}((f  - f_n)_0) \leq \ow{Q}(f-f_n) \leq \|f - f_n\|_o \to 0, \text{ as }n \to \infty. $$
 Since $(f_n)$ converges in $\ell^2(X,m)$ (as $(f_n)$ is a $\|\cdot\|_{Q_q}$ Cauchy sequence) and since convergence in $(\ow{D},\|\cdot\|_o)$ implies pointwise convergence,
 we obtain $\|f-f_n\|_2 \to 0$, as $n \to \infty$. Thus, $f\in \ow{D}\cap\ell^2(X,m)=D(\QN)$ by definition of $\QN$.
 Moreover, the continuity of the trace map $\Tr:(\ow{D},\|\cdot\|_o)\to L^2(\partial_h X,\mu)$, c.f. Corollary~\ref{corollary: properties of trace map}, yields $\Tr f_n \to \Tr f$ in $L^2(\partial_h X, \mu)$. From the lower semicontinuity of $q$ on $L^2(\partial_h X,\mu)$, we infer
 $$q(\Tr f - \Tr f_n) \leq \liminf_{k \to \infty} q(\Tr f_k - \Tr f_n) \leq \liminf_{k \to \infty} \|f_k - f_n\|_{Q_q}^2.$$
  Altogether, this implies $\|f - f_n\|_{Q_q} \to 0$, as $n\to \infty$, and so the closedness of $Q_q$ is proven.

  Now assume that $m$ is finite. We first show $q = \Tr Q_q$. Since bounded functions are dense in the domains of Dirichlet forms by Proposition~\ref{proposition:algebraic and order properties of domains}, it suffices to prove
  $$D(q) \cap L^\infty(\partial_h X,\mu) = D(\Tr Q_q) \cap L^\infty(\partial_h X,\mu)$$
  and that $q$ and $\Tr Q_q$ agree on these spaces.

  Let $\varphi \in D(\Tr Q_q) \cap L^\infty(\partial_h X,\mu)$.
  By definition of the trace form we have  $H_\varphi \in D(Q_{q,{\rm e}})$. Since $\varphi$ is bounded, we infer
  $$|H_\varphi(y)|=\left|\int_{\partial_h X}\varphi~d\mu_y\right|\leq\|\varphi\|_\infty\mu_y(\partial_h X)\leq\|\varphi\|_\infty$$ for every $y\in X$.
  Hence,
  $H_\varphi \in \ell^\infty(X)$ and so $H_\varphi \in \ell^2(X,m)$ by the finiteness of $m$. The characterization of form domains in terms of extended form domains, see \cite[Theorem~1.1.5]{CF}, implies
  $$D(Q_{q}) = D(Q_{q,{\rm e}}) \cap \ell^2(X,m)$$
  and we obtain $H_\varphi \in D(Q_q)$. Since $\varphi\in D(\Tr Q_q)\subseteq \Tr\ow{D}$  an application of Lemma~\ref{proposition:properties of trace} (d) yields
  $\varphi = \Tr H_\varphi$. By definition of $Q_q$ we have $\Tr H_\varphi\in D(q)$ and, therefore, $\varphi\in D(q)$. Hence, we deduce
  $ D(\Tr Q_q) \cap L^\infty(\partial_h X,\mu)\subseteq D(q) \cap L^\infty(\partial_h X,\mu)$. Moreover,
  $$q(\varphi) = q(\Tr H_\varphi) = Q_q(H_\varphi) = \Tr Q_q (\varphi).$$
   This proves that for bounded functions $q$ is an extension of $\Tr Q_q$.

   Let $\varphi \in D(q) \cap L^\infty(\partial_h X,\mu)$.  The inequality $q \geq q^{DN}$ yields $\varphi \in D(q^{DN})$,
   which implies $H_{\varphi} \in D(\QNe) \subseteq \ow{D}$   by definition of $q^{DN}$.  As seen above, the boundedness of $\varphi$ implies $H_\varphi \in \ow{D} \cap \ell^\infty(X)$
   and so the finiteness of $m$ implies $H_\varphi \in   \ow{D} \cap \ell^2(X,m) = D(\QN)$ by definition of $\QN$. Since we have $\Tr H_\varphi = \varphi \in D(q)$
   (by Lemma~\ref{proposition:properties of trace} (d) and $D(q) \subseteq D(q^{DN}) \subseteq \Tr \ow{D}$), we obtain $H_\varphi \in D(Q_q)$.
   Hence, $\varphi = \Tr H_\varphi \in \Tr D(Q_q)\subseteq \Tr D(Q_{q,e})=D(\Tr Q_q)$ by definition of $\Tr Q_q$.  This concludes the proof of (b).

  \end{proof}

\begin{proof}[Proof of Corollary~\ref{coro-parametrization}]
The injectivity of the map $\Tr$ follows from   (a) of the theorem
and its surjectivity (for finite $m$) follows from  (b) of the
theorem.
\end{proof}

 \section{A toy example}\label{toy}
 In this section we discuss the example of a star graph consisting of copies of $\IN$ attached to a center. In order to get a finite Royden boundary we assume that the inverses of the
 weights along the rays are summable. We give an explicit formula for $q^{DN}$ in the case that all rays look the same. However, we explain how the case of different rays can be treated at the end of
 this section.

Let $N>1$ be a natural number. We consider a set $X$ given by $N$ copies of $\IN$ denoted by $\IN_j$, $j\in\{1,\ldots,N\}$, and a point $0$. Denote the $k$-th element of the $j$-th copy of $\IN$ by $k_j$, $k\in\IN, j\in\{1,\ldots,N\}$.
We set $0_j:=0$ for every $j$. Let $(b_i)_{i\in\IN}$ be a sequence of positive numbers such that $$B:=b_1\sum_{i=1}^\infty \frac{1}{b_i}<\infty.$$
We induce a graph structure $(b,0)$ on $X$ via $$b((k-1)_j,k_j)=b(k_j,(k-1)_j):=b_k,$$ $k\in\IN,j\in\{1,\ldots,N\}$, and $b(x,y)=0$ otherwise.
Thus, we have a star graph with $N$ rays and symmetric weights on the edges.
\cite[Theorem 6.18 and Theorem 6.34]{soardi} yields that the harmonic boundary $\partial_h X$ consists of exactly $N$ points, denoted by $\infty_j$, $j\in\{1,\ldots,N\}$,
where each boundary point $\infty_j$  belongs to one ray $\IN_j$, i.e., every sequence of vertices that approximates the point $\infty_j$ is eventually contained in the ray $\IN_j$.
Moreover, we infer that  $\partial X=\partial_h X$ holds by \cite[Example 4.6]{canon},  \cite[Corollary 2.3 and Theorem 4.2]{uniform}.

Our goal in this section is to show how one can explicitly compute  the Dirichlet to Neumann form $q^{DN}$.
At first, we construct a basis of the space of harmonic functions in $\ow{D}$. Using \cite[Lemma 6.4]{soardi} we obtain that the space of harmonic functions is $N$-dimensional.
Since $c\equiv 0$, the constant functions are harmonic. Let $h_1\equiv 1\in\ow{D}$.
For the other $N-1$ base elements we define $h_j(0)=0$, $h_j(k_1)=-\sum_{l=1}^k\frac{b_1}{b_l}$, $h_j(k_j)=\sum_{l=1}^k\frac{b_1}{b_l}$, $k\in\IN$ and $h_j(k_i)=0$ for $i\not=1,i\not=j$ and $k\in\IN$. Hence, the function $h_j$ is supported on $\IN_1\cup\IN_j$.
The functions $h_1,\ldots,h_n$ are obviously linearly independent. One can easily check that they are harmonic and in $\ow{D}$.
Moreover, we infer $h_j(\infty_j)=-h_1(\infty_j)=B$ and
$h_j(\infty_k)=0$  otherwise, $j\in\{2,\ldots,N\}$.

Let $\varphi\in C(\partial_h X)$ be arbitrary and let $H_\varphi$ be the harmonic extension. Then, there are $\lambda_1,\ldots,\lambda_N$ such that $\lambda_1h_1+\ldots+\lambda_N h_N=H_\varphi$.
Applying the preliminary trace $\gamma_0$ yields the equations
\begin{align*}
\varphi(\infty_j)=\lambda_j \sum_{l=1}^\infty \frac{b_1}{b_l}+\lambda_1=\lambda_j B+\lambda_1,\quad j\in\{2,\ldots, N\}
\end{align*}
and
\[\varphi(\infty_1)=-(\lambda_2+\ldots+\lambda_N) \sum_{l=1}^\infty \frac{b_1}{b_l}+\lambda_1=(\lambda_2+\ldots+\lambda_N)B+\lambda_1.\]
This system has the solution \[\lambda_1=\frac{\varphi(\infty_1)+\ldots+\varphi(\infty_N)}{N}, \lambda_j=\frac{\varphi(\infty_j)-\lambda_1}{B}.\]
Therefore, we can construct the harmonic extension $H_\varphi$ and, hence, the harmonic measures $\mu_x$. The harmonic measure at $0$ has a particularly simple structure, it is given by
$\mu_0(\{\infty_j\})=\frac{1}{N}$.

Now let $m$ be a finite measure given by a function $m:X\to (0,\infty)$. Then, every function in $\ow{D}$ is in $\ell^2(X,m)$ as well, since every function in $\ow{D}$ is bounded,
c.f. \cite[Example~4.6]{canon}. Hence, we infer $D(\NeuDir)=D(Q_e^{(N)})=\ow{D}$. Then, the Dirichlet to Neumann form $q^{DN}$ is defined on
$D(q^{DN})=\Tr D(Q_e^{(N)})=\Tr \ow{D}\subseteq L^2(\partial_h X,\mu_0)=C(\partial_h X)$ and given by \[q^{DN}(\varphi)=\sum_{i=2}^N\sum_{j=2}^N \lambda_i\lambda_j \ow{Q}(h_i,h_j),\] where the $\lambda_i$ are defined as above. Furthermore, we have
\[\ow{Q}(h_i,h_j)=\frac122\sum_{i=1}^\infty b_i \frac{b_1^2}{b_i^2}=b_1 B,\quad i\not=j,\]  and
\[\ow{Q}(h_i)=\frac{1}{2}\left(2\sum_{i=1}^\infty b_i \frac{b_1^2}{b_i^2}+2\sum_{i=1}^\infty b_i \frac{(-b_1)^2}{b_i^2}\right)=2b_1B.\]
Hence, we infer \begin{align*}
q^{DN}(\varphi)=\frac{b_1}{B}\sum_{i=2}^N\sum_{j=2,j\not=i}^N (\varphi(\infty_i)-\lambda_1)(\varphi(\infty_j)-\lambda_1)+\frac{2b_1}{B}\sum_{i=2}^N (\varphi(\infty_i)-\lambda_1)^2
\end{align*}
with $\lambda_1=\frac{\varphi(\infty_1)+\ldots+\varphi(\infty_N)}{N}$ as above.
Further simplifying the right hand side yields
$$q^{DN}(\varphi)=\frac{b_1}{2BN}\sum_{i,j=1}^N (\varphi(\infty_i)-\varphi(\infty_j))^2.$$

If we are given a function $u$ in $\ow{D}$, then we can calculate $\Tr u$ by taking the limit along the rays. Then we can compute $u_h=H_{\Tr u}$ as above and, hence,
we can compute $u_0$ as well.  Therefore, in this easy example we can construct all parts of the decomposition  $\NeuDir(u)=\DirDir(u_0)+q^{DN}(\Tr u)$.
\begin{remarks}
 This calculation can easily be extended to the case where each ray has a separate system $(b_k^{(j)})_k,$ $j\in\{1,\ldots, N\}$, such that $\sum_{k=1}^\infty\frac{1}{b_k^{(j)}}<\infty$.
 In this case a basis for the harmonic functions in $\ow{D}$ is given by $h_1\equiv 1$ and  $h_j(k_1)=-\sum_{l=1}^k\frac{b_1^{(1)}}{b_l^{(1)}}$,
 $h_j(k_j)=\sum_{l=1}^k\frac{b_1^{(1)}}{b_l^{(j)}}$, $k\in\IN$, and $h_j(k_i)=0$ for $i\not=1,i\not=j$ and $k\in\IN$. However, in this situation the
 $(\lambda_k)_{k=1,\ldots, N}$ have a more difficult structure.
\end{remarks}

\section{A example of a form  extending $Q^{(D)}$ but  not satisfying $Q \geq Q^{(N)}$}\label{appendix:beispiel}

In this appendix we give an explicit example for a weighted graph $(b,c)$ and a form $Q$ that extends $Q^{(D)}$ but does not satisfy $Q \geq \QN$. In principle, we use the same arguments as in \cite[Example~3.32]{Sch}. We let $X = \IZ^d$ the integer lattice of dimension $d$ and choose standard weights
$$b:\IZ^d \times \IZ^d \to \{0,1\}, \, b(x,y) = \begin{cases}
1 &\text{if } |x - y| = 1,\\
0 &\text{else}.
\end{cases}
$$
For a detailed proof of the following claims we refer to the discussion in \cite[Section~6]{uniform}.  It follows from the  Liouville property of $\IZ^d$ that any harmonic function of finite energy with respect to $(b,0)$ is constant. If $d \geq 3$, the graph $(b,0)$ is transient and  the Royden decomposition theorem implies $\ow{D} = D_0 + \IR \cdot 1$, where the sum is a direct sum of vector spaces. As a Cayley graph of a group, $(b,0)$ is even uniformly transient when $d \geq 3$, i.e., $D_0 \subseteq C_0(\IZ^d)$ (here $C_0(\IZ^d)$ denotes the functions $f:\IZ^d \to \IR$ that satisfy $\lim_{|x| \to \infty} f(x) = 0)$. In particular, the Royden boundary of $(b,0)$ consists of one point $\infty$, and the limit $\lim_{|x| \to \infty} f(x)$ exists for any $f \in \ow{D}$ and equals $f(\infty)$.

Let $c = \delta_0$, i.e., $c(0) = 1$ and $c(x) = 0$, otherwise. The space $\ow{D}$ is the same for the graphs $(b,0)$ and $(b,c)$. Likewise, $D_0$ does not see $c$, and the Royden compactifications of $(b,c)$ and $(b,0)$ agree and are given by the one-point compactification $\IZ^d \cup \{\infty\}$. We choose a finite measure $m$ on $\IZ^d$.  It follows from \cite[Theorem~1.1.5]{CF} that $D(Q_{(b,c)}^{(D)}) = (Q_{(b,c),\, {\rm e}}^{(D)}) \cap \ell^2(\IZ^d,m)$. Therefore, Lemma~\ref{lemma:extended dirichlet spaces of qd and qn} and the boundedness of functions in $D_0$ implies $D(Q_{(b,c)}^{(D)}) = D_0$. Moreover, the boundedness of functions in $\ow{D}$ yields $D(Q_{(b,c)}^{(N)}) = \ow{D}\cap \ell^2(\IZ^d,m) =\ow{D}$.

We define the quadratic form $Q$ on $\ell^2(\IZ^d,m)$ by letting $D(Q) = \ow{D}$ and
$$Q(f) := Q_{(b,0)}^{(N)}(f) + (f(\infty) - f(0))^2.$$
It is easily verified that $Q$ is a Dirichlet form.  For any $f \in D_0 = D(Q_{(b,c)}^{(D)})$ we have   $f(\infty) = 0$ and so
$$Q(f) =Q_{(b,0)}^{(N)}(f) + f(0)^2 = Q_{(b,c)}^{(D)}(f).$$
Moreover, $1 \in D(Q)$ and $Q(1) = 0 < 1 = Q^{(N)}_{(b,c)}(1)$. Therefore, $Q$ is a form that extends $Q_{(b,c)}^{(D)}$ but does not satisfy $Q \geq Q_{(b,c)}^{(N)}$.

The abstract reason for the existence of a form $Q$ that extends $\QD$ but does not satisfy $Q \geq \QN$ if $c \neq 0$ is the following. The killing $c$ on functions that vanish on the Royden boundary can  either be induced by the killing itself or by jumps from the inside to the Royden boundary. In the previous example this manifests in the fact that the terms
$$(f(0) - f(\infty))^2 \text{ and } f(0)^2$$
agree for functions that vanish at the point $\infty$.

\phantomsection

 \appendix

 \section{A lemma on lower semicontinuous forms}

 The following lemma is a consequence of \cite[Lemma~1.45]{Sch}. We include the proof for the convenience of the reader.

   \begin{lemma} \label{lemma:q-wekly convergent sequences}
    Let $X$ be a topological space and let $Q$ be a quadratic form on $D(Q)\subseteq C(X)$ which is lower semicontinuous with respect to pointwise convergence. Then every $Q$-bounded sequence in $D(Q)$ that converges pointwise to $u$ converges $Q$-weakly  to $u$.
   \end{lemma}
   \begin{proof}
    The lower semicontinuity of $Q$ and the $Q$-boundedness of $(u_i)$ imply $u \in D(Q)$. Hence, we can assume $u = 0$.
    The $Q$-boundedness of $(u_i)$ implies that for each $v \in D(Q)$ we have
    $$- \infty < \liminf_{i\to\infty} Q(u_i,v) \leq \limsup_{i\to\infty} Q(u_i,v) < \infty. $$
    Let $M \geq 0$ such that $Q(u_i) \leq M$ for each $i\in\IN$.  For  $\alpha > 0$ and $v \in D(Q)$ we obtain $v-\alpha u_i \to v$ pointwise.  The lower semicontinuity of $Q$   yields
    \begin{align*}
    Q(v) &\leq \liminf_{i\to\infty} Q(v - \alpha u_i)\\
    &= \liminf_{i\to\infty} \left(Q(v) - 2\alpha Q(u_i,v) + \alpha^2 Q(u_i) \right) \\
    &\leq \liminf_{i\to\infty} \left(Q(v) - 2\alpha Q(u_i,v) + \alpha^2 M \right) \\
    &=Q(v) -2\alpha \limsup_{i\to\infty} Q(u_i,v) + \alpha^2 M.
    \end{align*}
    Hence, for all $\alpha > 0$ we obtain $ 2\limsup_i Q(u_i,v) \leq \alpha M$, which implies $\limsup_i Q(u_i,v) \leq 0$.  Since $v$ was arbitrary, we also have $\limsup_i Q(u_i,-v) \leq 0$ and conclude
    $$0 \leq  \liminf_{i\to\infty} Q(u_i,v) \leq \limsup_{i\to\infty} Q(u_i,v) \leq 0.$$
    This finishes the proof. \end{proof}
%

 \section{A characterization of transience}\label{section:Char_Trans}
 The following theorem is a well-known characterization of transience in terms of associated function spaces and the existence of Green's function. Recall that a graph is called transient
 if for all $f\in D_0$ the equality  $\ow{Q}(f)=0$  implies $f=0$.
\begin{theorem}[Characterization of transience]\label{CharTrans}
 Let $(b,c)$ be a graph over $X$. The following assertions are equivalent.
 \begin{itemize}
  \item[(i)] For all  $o\in X$ the functionals $\|\cdot\|_o$ and $\widetilde{Q}^{\frac{1}{2}}$ are equivalent norms on $C_c(X)$.
  \item[(ii)] For all $o\in X$ there exists $C_o> 0$ such that $|\varphi(o)|^2\leq C_o\widetilde{Q}(\varphi)$ for all $\varphi\in C_c(X)$.
  \item[(iii)] The graph $(b,c)$ is transient.
  \item[(iv)] For all $x \in X$ there exists $g_x \in \Dzero \cap \ell^\infty(X)$ such that for all $f \in \Dzero$
  $$f(x) = \ow{Q}(g_x,f).$$

 \end{itemize}
 Moreover, if one (and, thus, all) of the conditions above is satisfied, then for every $x\in X$ the function $g_x$ is bounded with $\|g_x\|_\infty=g_x(x)$.
\end{theorem}
\begin{proof}
 For the equivalence of (i), (ii) and (iii) see e.g.  \cite[Theorem~B.2]{uniform}. The implication (iv) $\Rightarrow$ (ii) is a consequence of the Cauchy-Schwarz inequality.

 (i) \& (ii) $\Rightarrow$ (iv): Assertion (i) implies that $\ow{Q}^{1/2}$ and $\|\cdot\|_o$ are equivalent norms on $\Dzero$, the closure of $C_c(X)$ in $(\ow{D},\|\cdot\|_o)$. Since the latter space is complete, $(\Dzero,\ow{Q}^{1/2})$ is a Hilbert space. Assertion (ii) states that for each $x \in X$ the functional
 $$\Dzero \to \IR,\, f \mapsto f(x)$$
 is continuous with respect to $\ow{Q}^{1/2}$. According to the Riesz representation theorem, for each $x \in X$ there exists a function $g_x \in \Dzero$ such that $f(x) = \ow{Q}(g_x,f)$ for all $f \in \Dzero$. For proving the boundedness of $g_x$  we note that $g_x(x) = \ow{Q}(g_x) \geq 0 $ and compute
 \begin{align*}
  \ow{Q}(|g_x| \wedge g_x(x) - g_x) &= \ow{Q} ( |g_x| \wedge g_x(x) ) + \ow{Q}(g_x) - 2\ow{Q}(g_x, |g_x| \wedge g_x(x)) \\
  &\leq \ow{Q} (g_x) + \ow{Q}(g_x) - 2\ow{Q}(g_x, |g_x| \wedge g_x(x))\\
  &= 2g_x(x) - 2|g_x(x)| \wedge g_x(x)\\
  &= 0.
 \end{align*}
 Here, we used the compatibility of $\ow{Q}$ with normal contractions. According to (ii), this implies $g_x = |g_x| \wedge g_x(x) $ and so we obtain $\|g_x\|_\infty = g_x(x).$ This finishes the proof.
\end{proof}

\begin{remarks}
  Obviously, $C_c(X)$ can be replaced by $\Dzero$ in the assertions (i) and (ii).
 \end{remarks}
 \begin{remarks}
  Assertion (ii)  shows that for a transient graph the $\Qgen$-convergence of a sequence in $\Dzero$ implies its pointwise convergence.
 \end{remarks}
 \begin{remarks}
 The functions $g_x, x \in X$, in assertion (iv)  are unique and the function $G:X \times X \to \IR, \, G(x,y):= g_x(y)$ is sometimes called the Green's function of $(b,c)$.
\end{remarks}

\section{The proof of Kasue's theorem for the Royden boundary}\label{appendix: proof kasue}

Here we prove Theorem~\ref{gammaLtwo}, which states that on
transient graphs with $\partial_h X\not=\emptyset$ the preliminary
trace map $\gamma_0$ is a continuous operator from
$\ow{D}\cap\ell^\infty(X)$ to $L^2(\partial_h X,\mu)$. This was
proven by Kasue for the Kuramochi boundary in \cite{kasue}. We
provide the corresponding proof for the Royden boundary, which
carries over almost line by line. We start with proving two
auxiliary lemmas.
\begin{lemma} \label{lemma:difference harmonic extension}
Let $(b,c)$ be a transient graph with $\partial_h X\not=\emptyset$. For all $\varphi\in\gamma_0(\BFEspace)$
$$H_{\varphi^2}-(H_\varphi)^2\in \Dzero\cap\linf.$$
\end{lemma}
\begin{proof}
Let $\varphi = \gamma_0 f$ with $f \in \ow{D} \cap \ell^\infty(X)$ be given. Since $\BFEspace$ is an algebra,
see Proposition~\ref{proposition:algebraic and order properties of domains},  we have $\varphi^2 = \gamma_0 f^2\in\gamma_0(\BFEspace)$ and according to Proposition~\ref{Hgamma},
this identity implies $H_{\varphi^2} = (f^2)_h\in \BFEspace$, where the boundedness follows from the Royden decomposition (Theorem~\ref{theorem:royden decomposition}).
The same arguments show $(H_\varphi)^2 = (f_h)^2 \in\BFEspace$.  An application of Proposition~\ref{Hgamma} yields
$$\gamma_0(H_{\varphi^2} - (H_\varphi)^2) = \varphi ^2 - (\gamma_0 H_\varphi )^2 = \varphi^2 - \varphi^2 = 0.$$
Thus, we obtain $H_{\varphi^2} - (H_\varphi)^2 \in \Dzero \cap \ell^\infty(X)$ from  Corollary~\ref{corollary:d0 as kernel of gamma0}.
\end{proof}

%

\begin{lemma}\label{HarmQuadrat}
Let $(b,c)$ a  graph over $X$. For all harmonic $h\in\BFEspace$ and all $x \in X$
$$-\LapUnw h^2(x)=\sum_{y\in X}b(x,y)(h(x)-h(y))^2+c(x)h(x)^2.$$
\end{lemma}
\begin{proof}
Let $h \in \ow{D}\cap \ell^\infty(X)$ harmonic and let $x \in X$. We compute
\begin{align*}
-\LapUnw &h^2(x) =\sum_{y\in X}b(x,y)(h(y)^2-h(x)^2)-c(x)h(x)^2\\
=&\sum_{y\in X}b(x,y)(h(x)-h(y))^2-c(x)h(x)^2 +2h(x)\sum_{y\in X}b(x,y)h(y)-2\sum_{y\in X}b(x,y)h (x)^2\\
=&\sum_{y\in X}b(x,y)(h(x)-h(y))^2-c(x)h^2(x) +2h(x)\sum_{y\in X}b(x,y)(h(y)-h(x)).
\end{align*}
Here, we used that $\sum_{y\in X}b(x,y)h(y)$ and $\sum_{y\in X}b(x,y)h(x)^2$ converge absolutely due to the boundedness of $h$ and since $\sum_{y\in X} b(x,y)<\infty$ for every $x\in X$.
Moreover, the  harmonicity of $h$ implies
$$\sum_{y\in X} b(x,y) (h(y)-h(x))=c(x)h(x).$$
Combining this observation with the previous computation finishes the proof.
\end{proof}

\begin{proof}[Proof of Theorem~\ref{gammaLtwo}]
Since $\mu$ is a harmonic measure, there exists an $x \in X$ such that $\mu = \mu_x$. By definition of $H_\varphi$ we obtain
$$\int_{\partial_h X} (\gamma_0 f)^2 \, d \mu = \int_{\partial_h X} (\gamma_0 f)^2 \, d \mu_x = H_{(\gamma_0 f)^2}(x) = (H_{(\gamma_0 f)^2}(x)  - (H_{\gamma_0 f}(x))^2) +  (H_{\gamma_0 f}(x))^2.$$
We start with the latter term.
According to Proposition~\ref{Hgamma}, the function $H_{\gamma_0 f} $ satisfies
$$H_{\gamma_0 f}(x) = f_h(x) = f(x) - f_0(x).$$
Since evaluating functions at a point is continuous on $(\ow{D},\|\cdot\|_o)$, see Proposition~\ref{lemma:properties of dtilde}, there exists a constant $C_1 \geq 0$ such that $|f(x)| \leq C_1 \|f\|_o$. Moreover, by the transience of $(b,c)$ and Theorem~\ref{CharTrans}, there exists a constant $C_2 \geq 0$ such that
$$|f_0(x)| \leq C_2 \ow{Q}(f_0)^{1/2} \leq  C_2 \ow{Q}(f )^{1/2} \leq C_2 \|f\|_o.$$
Hence, we infer, for a constant $C>0$, $$(H_{\gamma_0 f(x)})^2\leq C\|f\|_o^2.$$

To finish the proof it remains to consider the expression $H_{(\gamma_0 f)^2}(x)  - (H_{\gamma_0 f}(x))^2$. To this end, we employ Theorem~\ref{CharTrans} and let $g_x \in \Dzero\cap \ell^\infty(X)$ such that for all $g \in \Dzero$ we have
$$g(x) = \ow{Q}(g_x,g).$$
Since $H_{(\gamma_0 f)^2}  - (H_{\gamma_0 f})^2 \in \Dzero \cap \ell^\infty(X)$ by Lemma~\ref{lemma:difference harmonic extension}, we obtain
$$H_{(\gamma_0 f)^2}(x)  - (H_{\gamma_0 f}(x))^2 = \ow{Q}(g_x, H_{(\gamma_0 f)^2}  - (H_{\gamma_0 f})^2).$$
The harmonicity of $H_{(\gamma_0 f)^2}$, Lemma~\ref{HarmQuadrat} and the boundedness of $g_x$ yield
\begin{align*}
 \sum_{z \in X}   |g_x(z) \LapUnw (H_{(\gamma_0 f)^2}&-(H_{\gamma_0 f})^2) (z)| =  \sum_{z \in X}   |g_x(z)||\LapUnw(H_{\gamma_0 f})^2(z)| \\
 &= \sum_{z \in X} |g_x(z)| \left( \sum_{y\in X} b(z,y)(H_{\gamma_0 f}(z)-H_{\gamma_0 f}(y))^2+c(z)(H_{\gamma_0 f}(z))^2\right) \\
 &\leq \|g_x\|_\infty \ow{Q}(H_{\gamma_0 f})\\
 &=  \|g_x\|_\infty \ow{Q}(f_h)\\
 &\leq \|g_x\|_\infty \|f\|^2_o.
\end{align*}
This inequality shows that we can apply Green's formula, Proposition~\ref{Green}, to the functions $g_x$ and $H_{(\gamma_0 f)^2}  - (H_{\gamma_0 f})^2$. Combining it with the previous considerations yields
\begin{align*}
 |H_{(\gamma_0 f)^2}(x)  - (H_{\gamma_0 f}(x))^2| &= |\ow{Q}(g_x, H_{(\gamma_0 f)^2}  - (H_{\gamma_0 f})^2)| \\
 &= \left| \sum_{z \in X}   g_x(z) \LapUnw (H_{(\gamma_0 f)^2}-(H_{\gamma_0 f})^2) (z) \right|\\
 &\leq \|g_x\|_\infty \|f\|^2_o.
\end{align*}
This finishes the proof.
\end{proof}

\bibliography{Literature}{}

\newcommand{\etalchar}[1]{$^{#1}$}
\def\cprime{$'$}
\begin{thebibliography}{HKMW13}

\bibitem[AW03]{ArendtWarma}
Wolfgang Arendt and Mahamadi Warma.
\newblock Dirichlet and {N}eumann boundary conditions: {W}hat is in between?
\newblock {\em J. Evol. Equ.}, 3(1):119--135, 2003.
\newblock Dedicated to Philippe B\'enilan.

\bibitem[BCP68]{bony}
Jean-Michel Bony, Philippe Courr{\`e}ge, and Pierre Priouret.
\newblock {Semi-groupes de {F}eller sur une vari{\'e}t{\'e} {\`a} bord compacte
  et probl{\`e}mes aux limites int{\'e}gro-diff{\'e}rentiels du second ordre
  donnant lieu au principe du maximum}.
\newblock {\em Ann. Inst. Fourier (Grenoble)}, 18(fasc. 2):369--521 (1969),
  1968.

\bibitem[BGW09]{BGW}
B.~M. Brown, G.~Grubb, and I.~G. Wood.
\newblock {M-functions for closed extensions of adjoint pairs of operators with
  applications to elliptic boundary problems}.
\newblock {\em Math. Nachr.}, 282(3):314--347, 2009.

\bibitem[BMNW08]{BMN}
Malcolm Brown, Marco Marletta, Serguei Naboko, and Ian Wood.
\newblock {Boundary triplets and {$M$}-functions for non-selfadjoint operators,
  with applications to elliptic {PDE}s and block operator matrices}.
\newblock {\em J. Lond. Math. Soc. (2)}, 77(3):700--718, 2008.

\bibitem[CF12]{CF}
Zhen-Qing Chen and Masatoshi Fukushima.
\newblock {\em {Symmetric {M}arkov processes, time change, and boundary
  theory}}, volume~35 of {\em {London Mathematical Society Monographs Series}}.
\newblock Princeton University Press, Princeton, NJ, 2012.

\bibitem[CTHT11a]{CdVTHT2}
Yves {Colin de Verdi{\`e}re}, Nabila Torki-Hamza, and Fran\c{c}oise Truc.
\newblock {Essential self-adjointness for combinatorial {S}chr{\"o}dinger
  operators {II}---metrically non complete graphs}.
\newblock {\em Math. Phys. Anal. Geom.}, 14(1):21--38, 2011.

\bibitem[CTHT11b]{CdVTHT1}
Yves {Colin de Verdi{\`e}re}, Nabila Torki-Hamza, and Fran\c{c}oise Truc.
\newblock {Essential self-adjointness for combinatorial {S}chr{\"o}dinger
  operators {III}---{M}agnetic fields}.
\newblock {\em Ann. Fac. Sci. Toulouse Math. (6)}, 20(3):599--611, 2011.

\bibitem[Fel57]{feller}
William Feller.
\newblock {Generalized second order differential operators and their lateral
  conditions}.
\newblock {\em Illinois J. Math.}, 1:459--504, 1957.

\bibitem[FOT11]{FOT}
Masatoshi Fukushima, Yoichi Oshima, and Masayoshi Takeda.
\newblock {\em {Dirichlet forms and symmetric {M}arkov processes}}, volume~19
  of {\em {De Gruyter Studies in Mathematics}}.
\newblock Walter de Gruyter \& Co., Berlin, extended edition, 2011.

\bibitem[Fuk69]{Fuk}
Masatoshi Fukushima.
\newblock {On boundary conditions for multi-dimensional {B}rownian motions with
  symmetric resolvent densities}.
\newblock {\em J. Math. Soc. Japan}, 21:58--93, 1969.

\bibitem[GHK{\etalchar{+}}15]{canon}
Agelos Georgakopoulos, Sebastian Haeseler, Matthias Keller, Daniel Lenz, and
  Rados{\l}aw~K. Wojciechowski.
\newblock {Graphs of finite measure}.
\newblock {\em J. Math. Pures Appl. (9)}, 103(5):1093--1131, 2015.

\bibitem[GM11]{GM}
Fritz Gesztesy and Marius Mitrea.
\newblock {A description of all self-adjoint extensions of the Laplacian and
  Krein-type resolvent formulas on non-smooth domains}.
\newblock {\em J. Anal. Math.}, 113:53--172, 2011.

\bibitem[Gru08]{Gr}
G.~Grubb.
\newblock {Krein resolvent formulas for elliptic boundary problems in nonsmooth
  domains}.
\newblock {\em Rend. Semin. Mat. Univ. Politec. Torino}, 66(4):271--297, 2008.

\bibitem[HK11]{solu}
Sebastian Haeseler and Matthias Keller.
\newblock {Generalized solutions and spectrum for {D}irichlet forms on graphs}.
\newblock In {\em {Random walks, boundaries and spectra}}, volume~64 of {\em
  {Progr. Probab.}}, pages 181--199. Birkh{\"a}user/Springer Basel AG, Basel,
  2011.

\bibitem[HKLW12]{HKLW}
Sebastian Haeseler, Matthias Keller, Daniel Lenz, and Rados{\l}aw
  Wojciechowski.
\newblock {Laplacians on infinite graphs: {D}irichlet and {N}eumann boundary
  conditions}.
\newblock {\em J. Spectr. Theory}, 2(4):397--432, 2012.

\bibitem[HKMW13]{HKMW}
Xueping Huang, Matthias Keller, Jun Masamune, and Rados{\l}aw~K. Wojciechowski.
\newblock {A note on self-adjoint extensions of the {L}aplacian on weighted
  graphs}.
\newblock {\em J. Funct. Anal.}, 265(8):1556--1578, 2013.

\bibitem[JP11]{JP}
Palle E.~T. Jorgensen and Erin P.~J. Pearse.
\newblock {Resistance boundaries of infinite networks}.
\newblock In {\em {Random walks, boundaries and spectra}}, volume~64 of {\em
  {Progr. Probab.}}, pages 111--142. Birkh{\"a}user/Springer Basel AG, Basel,
  2011.

\bibitem[Kas10]{kasue}
Atsushi Kasue.
\newblock {Convergence of metric graphs and energy forms}.
\newblock {\em Rev. Mat. Iberoam.}, 26(2):367--448, 2010.

\bibitem[Kas17]{kasue2}
Atsushi Kasue.
\newblock Convergence of {D}irichlet forms induced on boundaries of transient
  networks.
\newblock {\em Potential Anal.}, 47(2):189--233, 2017.

\bibitem[KL12]{stoch}
Matthias Keller and Daniel Lenz.
\newblock {Dirichlet forms and stochastic completeness of graphs and
  subgraphs}.
\newblock {\em J. Reine Angew. Math.}, 666:189--223, 2012.

\bibitem[KLSW17]{uniform}
Matthias Keller, Daniel Lenz, Marcel Schmidt, and Rados\l~aw Wojciechowski.
\newblock Note on uniformly transient graphs.
\newblock {\em Rev. Mat. Iberoam.}, 33(3):831--860, 2017.

\bibitem[KLW16]{Ka-sing}
Shi-Lei Kong, Ka-Sing Lau, and Ting-Kam~Leonard Wong.
\newblock {Random walks and induced Dirichlet forms on self-similar sets}.
\newblock 2016.
\newblock {arXiv:1604.05440 }.

\bibitem[Mal10]{Mal}
M.~M. Malamud.
\newblock {Spectral theory of elliptic operators in exterior domains}.
\newblock {\em Russ. J. Math. Phys.}, 17(1):96--125, 2010.

\bibitem[Mil11]{Mil}
Ognjen Milatovic.
\newblock {Essential self-adjointness of magnetic {S}chr{\"o}dinger operators
  on locally finite graphs}.
\newblock {\em Integral Equations Operator Theory}, 71(1):13--27, 2011.

\bibitem[Pos08]{Pos}
Andrea Posilicano.
\newblock {Self-adjoint extensions of restrictions}.
\newblock {\em Oper. Matrices}, 2(4):483--506, 2008.

\bibitem[Pos14]{posi}
Andrea Posilicano.
\newblock {Markovian extensions of symmetric second order elliptic differential
  operators}.
\newblock {\em Math. Nachr.}, 287(16):1848--1885, 2014.

\bibitem[PR09]{PR}
Andrea Posilicano and Luca Raimondi.
\newblock {Krein's resolvent formula for self-adjoint extensions of symmetric
  second-order elliptic differential operators}.
\newblock {\em J. Phys. A}, 42(1):015204, 11, 2009.

\bibitem[Ryz07]{Ryz}
Vladimir Ryzhov.
\newblock {A general boundary value problem and its {W}eyl function}.
\newblock {\em Opuscula Math.}, 27(2):305--331, 2007.

\bibitem[Sch99a]{Schmu2}
Byron Schmuland.
\newblock {Extended {D}irichlet spaces}.
\newblock {\em C. R. Math. Acad. Sci. Soc. R. Can.}, 21(4):146--152, 1999.

\bibitem[Sch99b]{Schmu}
Byron Schmuland.
\newblock {Positivity preserving forms have the {F}atou property}.
\newblock {\em Potential Anal.}, 10(4):373--378, 1999.

\bibitem[Sch16]{Sch}
Marcel Schmidt.
\newblock {\em {Energy forms}}.
\newblock PhD thesis, Friedrich-Schiller-Universit{\"a}t Jena, 2016.

\bibitem[Sch17]{Sch17}
Marcel Schmidt.
\newblock Global properties of {D}irichlet forms on discrete spaces.
\newblock {\em Dissertationes Math. (Rozprawy Mat.)}, 522:43, 2017.

\bibitem[Sil74a]{Sil}
Martin~L. Silverstein.
\newblock Classification of stable symmetric {M}arkov chains.
\newblock {\em Indiana Univ. Math. J.}, 24:29--77, 1974.

\bibitem[Sil74b]{Si1}
Martin~L. Silverstein.
\newblock {\em {Symmetric {M}arkov processes}}.
\newblock {Lecture Notes in Mathematics, Vol. 426}. Springer-Verlag, Berlin-New
  York, 1974.

\bibitem[Soa94]{soardi}
Paolo~M. Soardi.
\newblock {\em {Potential theory on infinite networks}}, volume 1590 of {\em
  {Lecture Notes in Mathematics}}.
\newblock Springer-Verlag, Berlin, 1994.

\bibitem[Tai04]{taira}
Kazuaki Taira.
\newblock {\em {Semigroups, boundary value problems and {M}arkov processes}}.
\newblock {Springer Monographs in Mathematics}. Springer-Verlag, Berlin, 2004.

\bibitem[TH10]{TH}
Nabila Torki-Hamza.
\newblock {Laplaciens de graphes infinis ({I}-graphes) m{\'e}triquement
  complets}.
\newblock {\em Confluentes Math.}, 2(3):333--350, 2010.

\bibitem[Uen60]{ueno}
Tadashi Ueno.
\newblock {The diffusion satisfying {W}entzell's boundary condition and the
  {M}arkov process on the boundary. {I}, {II}}.
\newblock {\em Proc. Japan Acad.}, 36:533--538, 625--629, 1960.

\bibitem[Ven59]{wentzell}
A.~D. Ventcel{\cprime}.
\newblock {On boundary conditions for multi-dimensional diffusion processes}.
\newblock {\em Theor. Probability Appl.}, 4:164--177, 1959.

\bibitem[Woj08]{WojDiss}
RK~Wojciechowski.
\newblock {\em {Stochastic completeness of graphs}}.
\newblock PhD thesis, Thesis (Ph.D.)-City University of New York., 2008.

\end{thebibliography}
\bibliographystyle{alpha}
\end{document}